\documentclass[a4paper,twoside,english]{amsart}
\usepackage{graphicx}  
\usepackage{amsmath,amssymb,amsthm, amscd}
\usepackage{amssymb,fge}
\usepackage{tikz-cd}
\usepackage[T1]{fontenc}
\usepackage[latin9]{inputenc}
\usepackage{babel}
\usepackage{verbatim}
\makeatletter
\let\@wraptoccontribs\wraptoccontribs
\makeatother

\usepackage{lipsum}

\newcommand\blfootnote[1]{%
  \begingroup
  \renewcommand\thefootnote{}\footnote{#1}%
  \addtocounter{footnote}{-1}%
  \endgroup
}

\usepackage{color}

\usepackage[left=3.2cm,right=3.5cm,top=3cm,bottom=3cm]{geometry}
\usepackage{indentfirst}           
\usepackage{underscore}
\usepackage{enumitem}
\usepackage{relsize}
\usepackage{varwidth}
\usepackage{mathtools} 
\usepackage{hyperref}
\hypersetup{
    colorlinks=true,
    linkcolor=blue,
    filecolor=magenta,      
    urlcolor=cyan,
}
\usepackage{amssymb}
\usepackage[all]{xy}
\makeatletter
\renewcommand*\env@matrix[1][*\c@MaxMatrixCols c]{%
  \hskip -\arraycolsep
  \let\@ifnextchar\new@ifnextchar
  \array{#1}}
\makeatother

\newcommand\Seq{\text{Seq}}

\def\CVD{{\hfill\hfil{\lower 2pt\hbox{\vrule\vbox to 7pt
{\hrule width  5pt\varphifill\hrule}\varphirule}}}\par}

\DeclareMathOperator{\Pic}{Pic}
\DeclareMathOperator{\Orb}{Orb}

\newcommand{\mysetminus}{\mathbin{\fgebackslash}}

\newtheorem{theorem}{Theorem}[section]
\newtheorem{lemma}[theorem]{Lemma}

\newtheorem{proposition-definition}[theorem]{Proposition-Definition}
\newtheorem{corollary}[theorem]{Corollary}

\theoremstyle{definition}
\newtheorem{definition}[theorem]{Definition}
\theoremstyle{remark}
\newtheorem{remark}{Remark}
\newtheorem{example}{Example}

\theoremstyle{theorem}

\theoremstyle{remark}

\newcommand*{\Scale}[2][4]{\scalebox{#1}{$#2$}}%

\title[Semigroup orbit counts]{Orbit counting in polarized dynamical systems}   
\author[Wade Hindes]{Wade Hindes}
\begin{document} 
\maketitle
\begin{abstract} We extend recent orbit counts for finitely generated semigroups acting on $\mathbb{P}^N$ to certain infinitely generated, polarized semigroups acting on projective varieties. We then apply these results to semigroup orbits generated by some infinite sets of unicritical polynomials.  
\end{abstract}   
\section{Introduction} 
Given a projective variety $V$, a height function $H:V\rightarrow\mathbb{R}$, and a subset $X\subseteq V$ of some interest, it is often useful to understand the growth rate of the number of points in $X$ of bounded height. For instance, if $X$ is a set or rational or integral points, then this growth rate is known to encode many interesting arithmetic and geometric invariants of $V$ (like its  dimension, genus, or rank of an associated Mordell-Weil group); for examples, see \cite{BP,Faltings,HeathBrown1,HeathBrown2,Schanuel}. Likewise, when $X$ is a dynamical orbit generated by a collection of self maps of $V$, then the growth rate on the number of points in $X$ of bounded height frequently detects dynamical degrees \cite{Me:dyndeg,KawaguchiSilverman} as well as other invariants \cite{K3,Markoff}. In this paper, we take up this dynamical orbit counting problem, generalizing the main results from \cite{Me:monoid} in two ways: first we allow \emph{infinitely generated} semigroups, and second we allow $V$ to be any projective variety with a polarizable set of maps (not necessarily $\mathbb{P}^N$) . As an application, we obtain precise estimates for the number of points of bounded height in many new orbits, including those generated by the (possibly infinite) set of polynomials of the form $z^q+1$, where $q$ is a Mersenne prime; see Corollary \ref{cor:amusing} below.

To describe our results, we fix some notation. Let $K$ be a global field and let $V$ be a projective variety defined over $K$. Then, given a collection $S$ of endomorphisms of $V$, we define $M_S$ to be the semigroup generated by $S$ under composition together with the identity function (formally a monoid). In particular, to a point $P\in V$ we associate an \emph{orbit}, 
\[\Orb_S(P)=\{f(P)\,:\, f\in M_S\},\]
and study the growth rate of the number of points in $\Orb_S(P)$ of bounded height. There is very little known about this problem in general, even for sets $S$ of polynomials in one-variable. So, as in the case of iterating a single morphism \cite{Call-Silverman}, we make this problem more tractable by assuming that our dynamical system is \emph{polarized}: there is a divisor class $\eta\in\Pic(V)\otimes\mathbb{R}$ and a collection of real numbers $d_\phi$ such that $\phi^{*}(\eta)=d_\phi\eta$ for all $\phi\in S$. For in this case, there is a minimal constant $C(V,\eta,\phi)$ such that   
\[\big|h_\eta(\phi(P))-d_\phi h_\eta(P)\big| \leq C(V,\eta,\phi)\] 
for all $P\in V$; here $h_\eta$ is a height function associated to $\eta$. Then, in order to generalize the techniques in \cite{Me:monoid} for finitely generated semigroups acting on $\mathbb{P}^N$ to infinitely generated polarized semigroups acting on arbitrary projective varieties, we define the following conditions:        
\begin{definition}\label{def:htcontrolled,uniformdiscrete} A set of endomorphisms $S$ on a projective variety $V$ is called \emph{height controlled} with respect to a divisor class $\eta\in\Pic(V)\otimes\mathbb{R}$ if: \vspace{.1cm}   
\begin{enumerate} 
\item for all $\phi\in S$ there exists a real number $d_\phi>1$ such that $\phi^*(\eta)=d_\phi\eta$,\vspace{.1cm}   
\item the corresponding height constants are bounded: $\sup_{\phi\in S}C(V,\eta,\phi)$ is finite. \vspace{.1cm} 
\end{enumerate}
Moreover, if in addition the set of real numbers $T:=\{\log(d_\phi)\,:\phi\in S\}$ is \emph{uniformly discrete}, i.e., there exists a constant $\delta_T>0$ such that $\big|\log(d_\phi)-\log(d_\psi)\big|>\delta_T$ for all distinct $\phi,\psi$ in $S$, then we say that $S$ is \emph{\textbf{height controlled and uniformly log-discrete}} with respect to $\eta$.
\end{definition}
\begin{remark} The first condition, which we call height controlled but has also been called \emph{bounded}, has been used in several places to study both semigroup and random sequential orbits; see, for instance, \cite{Me:stochastic,Me:dyndeg,Kawaguchi,Mello}. However the second condition is a new and technical one, allowing us to control the dominant poles of some associated meromorphic generating functions; see Section \ref{sec:countingfunctions} for details. 
\end{remark} 
\begin{example} Let $S=\{\phi_1,\dots,\phi_s\}$ be any finite set of endomorphisms of $\mathbb{P}^N$. If $\deg(\phi)\geq2$ for all $\phi\in S$ and $\deg(\phi)\neq\deg(\psi)$ for all distinct $\phi,\psi\in S$, then $S$ is height controlled and uniformly log-discrete. More generally, let $S=\{\phi_1,\dots,\phi_s\}$ be any finite set of endomorphisms on a projective variety $V$ and assume that there is a divisor class $\eta\in\Pic(V)\otimes\mathbb{R}$ such that for each $\phi\in S$ there is a real number $d_\phi>1$ satisfying $\phi^*(\eta)=d_\phi\eta$ . If in addition $d_\phi\neq d_\psi$ for all distinct $\phi,\psi\in S$, then $S$ is height controlled and uniformly log-discrete with respect to $\eta$.      
\end{example}  
\begin{example}\label{eg:htcontrolledandlogdiscrete} For an infinite example, fix $c\in\overline{\mathbb{Q}}$ and $a,b\in\mathbb{N}$. Then the set of unicritical polynomials with constant term $c$ and degree $d=a^n+b$ for some $n\geq1$, 
\[S_{a,b,c}:=\{x^d+c\,:\,\text{$d=a^n+b$ for some $n\geq1$}\},\]
is a height controlled and uniformly log-discrete set of endomorphisms of $\mathbb{P}^1$.   
\end{example}  
With this terminology in place we are ready to state our first result, an upper bound for the number of points of bounded height in semigroup orbits generated by height controlled and uniformly log-discrete sets. Moreover, if in addition the semigroup is free, then we can give an associated lower bound. In particular, since generic compositional semigroups tend to be free (e.g., see \cite[\S4]{Me:monoid}), it is reasonable to expect that the bounds we give below apply more generally, providing a template for future work. In what follows, $H_\eta=\exp\circ h_\eta$ denotes the multiplicative Weil height associated to $\eta$. \vspace{.05cm}           
\begin{theorem}\label{thm:functioncount} Let $K$ be a global field, let $V$ be a projective variety over $K$, and let $S$ be a height controlled and uniformly log-discrete set of endomorphisms on $V$ with respect to a divisor class $\eta\in\Pic(V)\otimes\mathbb{R}$. Then the following statements hold: \vspace{.1cm} 
\begin{enumerate} 
\item There is a positive constant $b=b(S,\eta)$ and a constant $B_{S,\eta}$ such that 
\begin{equation*}
\#\big\{Q\in \Orb_{S}(P)\,:\, H_\eta(Q)\leq B\big\}\ll (\log B)^{b}\vspace{.1cm}
\end{equation*} 
for all $P\in V$ with $H_\eta(P)>B_{S,\eta}$. \vspace{.3cm} 
\item If $M_S$ is free (non-abelian), then for all $\epsilon>0$ there is a positive constant $b=b(S,\eta,\epsilon)$ and a constant $B_{S,\eta}$ such that  \vspace{.1cm} 
\begin{equation*}
(\log B)^{b}\ll\#\big\{f\in M_S\,:\, H_\eta(f(P))\leq B\big\}\ll (\log B)^{b+\epsilon}\vspace{.15cm}
\end{equation*} 
for all $P\in V$ with $H_\eta(P)>B_{S,\eta}$.  \vspace{.1cm} 
\end{enumerate} 
Here in both statements, the implicit constants depend on $P$, $S$, $\eta$ (and $\epsilon$ for statement (2)). \vspace{.1cm}  
\end{theorem}
Of course, one would like upper and lower bounds for the number of \emph{points} of bounded height in orbits, not just \emph{functions} yielding a bounded height relation. However, to do this for a given basepoint $P$ we need to control the number of distinct functions $f,g\in M_S$ that can agree at $P$, i.e., $f(P)=g(P)$. This is quite challenging in higher dimensions, even for finitely generated semigroups. On the other hand, it is tractable in dimension one due to Siegel's integral point theorem and the polynomial classification theorem of Bilu-Tichy \cite{Bilu}; essentially, it suffices to control the $\mathfrak{s}$-integral points on the \emph{finitely} many curves $f(x)=g(y)$ where $f$ and $g$ are in the generating set; see \cite{Me:monoid}. However, this problem remains quite difficult (even for $\mathbb{P}^1$) when $M_S$ is \emph{infinitely generated}, since there are infinitely many curves to consider. Nevertheless, there is a non-trivial case where it is possible, namely when we restrict ourselves to polynomials of the form $f(z)=z^q+1$ with prime exponent and constant term one. In this case, we combine Theorem \ref{thm:functioncount} above with the bounds in \cite{Brindza} on the $\mathfrak{s}$-integral solutions to the \emph{Catalan equations} $x^n-y^m=1$ to prove the following result. Below, $H: \mathbb{P}^1\rightarrow\mathbb{R}$ denotes the standard absolute and multiplicative Weil height. \vspace{.1cm}      

\begin{theorem}\label{thm:catalan} Let $K$ be a number field and let $w_K$ be the number of roots of unity in $K$. Given any uniformly log-discrete set $\mathfrak{q}$ of rational prime numbers,  define
\[S_{K,\mathfrak{q}}=\big\{z^q+1:\text{$q\in\mathfrak{q}$ and $\gcd(q,w_K)=1$}\big\}.\] 
Then for all $\epsilon>0$ there is a positive constant $b=b(K,\mathfrak{q},\epsilon)$ such that \vspace{.125cm} 
\begin{equation*}
(\log B)^{b}\ll\#\big\{Q\in \Orb_{S_{K,\mathfrak{q}}}(P)\,:\, H(Q)\leq B\big\}\ll (\log B)^{b+\epsilon}\vspace{.125cm}
\end{equation*} 
for all $P\in\mathbb{P}^1(K)$ with $H(P)>4$.  \vspace{.1cm}   
\end{theorem}
\begin{remark} As a first step in the proof of Theorem \ref{thm:catalan}, we show that $M_{S_{K,\mathfrak{q}}}$ is a free semigroup. In fact, we prove a stronger statement: if $k$ is any field of characteristic zero and \[S=\Big\{z^d+c\;:\; d\geq2,\; c\in k^*\Big\},\] then $M_S$ is free; see Theorem \ref{thm:free} in Section \ref{sec:free}.  
\end{remark} 
For a concrete and more classically looking arithmetic application of Theorem \ref{thm:catalan}, consider the Mersenne primes (i.e., those of the form $q=2^n-1$). Then, after applying some additional approximation techniques discussed in Remark \ref{rem:explicit} below, we prove an explicit version of Theorem \ref{thm:catalan} in this case; for simplicity, we state this result for $K=\mathbb{Q}$. \vspace{.1cm}       
\begin{corollary}\label{cor:amusing} Let $r\in\mathbb{Q}$ with $H(r)>4$ and let $\mathfrak{m}$ be the set of Mersenne primes. Then \vspace{.15cm} 
\[
\Scale[.94]{
(\log B)^{0.60839}\ll\#\bigg\{t=((r^{q_1}+1)^{q_2}\dots +1)^{q_s}+1\;\Big\vert\; \text{$q_1,\dots, q_s\in\mathfrak{m}$, $H(t)\leq B$}\bigg\}\ll (\log B)^{0.60872}}\,.  \vspace{.15cm} 
\] 
Here the implicit constants depend on the initial point $r\in\mathbb{Q}$. \vspace{.1cm} 
\end{corollary} 
At present, there are only 51 known Mersenne primes, and it is unknown whether or not there are infinitely many such primes - there are, however, probabilistic heuristics suggesting that there are infinitely many \cite[\S1.3.1]{Pomerance}. With this in mind, it is (to the author at least) amusing that we can give a reasonably good estimate for the number of points of bounded height in the orbits in Corollary \ref{cor:amusing} while remaining agnostic on the infinitude of Mersenne primes, an age-old question.       
\begin{remark} 
Certainly there exist many \emph{provably} infinite, uniformly log-discrete sets of primes (where Theorem \ref{thm:catalan} applies). To construct explicit examples, fix any prime $q$ and any real number $\delta>0$. Now let $q_0=q$ and define  
\[q_{n+1}=\min\big\{q\in\mathbb{N}: \text{$q$ is prime and $\log(q)>\log(q_n)+\delta$}\big\}.\] 
recursively. Then it is easy to check that $\mathfrak{q}(q_0,\delta)=\bigcup_{n=0}^\infty\{q_n\}$ is an infinite and uniformly log-discrete set of primes. However, we prefer to display the Mersenne primes as our chief example (even though we cannot be sure if they are infinite or not), due to their prominent place in the history of arithmetic.    
\end{remark} 

\section{Counting functions in semigroups}\label{sec:countingfunctions}
Similar to the case of semigroups acting on projective space \cite[Section 2]{Me:monoid}, the overall strategy used to prove Theorem \ref{thm:functioncount} is the following. First, aided by a suitable generalization of Tate's Telescoping Lemma, we relate the functions $f\in M_S$ yielding a bounded height relation $h(f(P))\leq B$ to the functions having bounded $\log$ degree (or more precisely bounded $\log$ $\eta$-degree, where $\eta$ is the common eigendivisor class in Definition \ref{def:htcontrolled,uniformdiscrete} above). Then, we approximate the number of functions of bounded log degree by the number of some related, restricted integer compositions (a standard combinatorial object); to do this we must choose suitable rational approximations for each $\log(d_\phi)$ and $\phi\in S$. Next we count restricted integer compositions by analyzing the dominant poles of some associated generating functions. Finally, we show that our estimates for the number of functions of bounded $\log$ degree (coming from integer compositions) become better and better as we refine our rational approximations for the $\log(d_\phi)$. 

\begin{remark} However, there is some added subtlety that arrises in this more general setting (i.e., infinitely generated polarized semigroups), largely coming from the fact that the aforementioned combinatorial generating functions are now meromorphic (as opposed to simply rational, as in the case when $S$ is finite \cite[\S 2]{Me:monoid}). 
\end{remark} 

With this sketch in place, we begin with a few auxiliary results, starting with a generalization of Tate's Telescoping Lemma. In what follows, given a function $f=\phi_1\circ\dots\circ\phi_m$ in $M_S$ for some $\phi_i\in S$, then we write $\deg_\eta(f)$ for the quantity $\prod_{i=1}^m d_{\phi_i}$ satisfying $f^*(\eta)=\deg_\eta(f)\eta$; here $\eta\in\Pic(V)\otimes\mathbb{R}$ is the common eigendivisor class given in Definition \ref{def:htcontrolled,uniformdiscrete} for height controlled sets. Then we have the following tool for relating heights and degrees.      
\begin{lemma}{\label{lem:tate}} Let $S$ be a height controlled set of endomorphisms of $V$, and let $d_S:=\inf_{\phi\in S}\{d_\phi\}$ and $C_S:=\sup_{\phi\in S}C(V,\eta,\phi)$ be the corresponding height controlling constants. Then 
\[\bigg|\frac{h_\eta(f(Q))}{\deg_\eta(f)}-h_\eta(Q)\bigg|\leq \frac{C_S}{d_S-1}\] 
for all $f\in M_S$ and all $Q\in V(\overline{K})$. 
\end{lemma} 
See \cite[Lemma 2.1]{Me:stochastic} for a proof of this fact. Having related heights and degrees, we may reduce counting $\#\{f\in M_S\,:\, h_\eta(f(P))\leq B\}$ to counting functions of bounded log degree. Moreover if $M_S$ is free and $d_\phi\neq d_\psi$ for all $\phi\neq\psi$ in $S$, then this is equivalent to counting finite sequences $\big(\log(d_{\phi_1}),\dots, \log(d_{\phi_m})\big)$ such that $\sum_{i=1}^m\log(d_{\phi_i})\leq B$. This type of problem is typical in combinatorics, except that the summands $\log(d_{\phi_i})$ are not positive integers. Nevertheless, by approximating the $\log(d_{\phi_i})$ by suitable rational numbers, we can estimate the number of sequences in the $\log(d_{\phi_i})$ with bounded sum by the number of some associated sequences of positive integers with bounded sum (for which there are combinatorial tools). 

To discuss the relevant combinatorial tools, we fix some notation. Let $T$ be a set of positive real numbers, and let $\Seq(T)$ be the set of \emph{finite sequences} of elements of $T$. Then for any sequence $\omega=(t_1,\dots, t_m)$ with $t_i\in T$, we define 
\begin{equation}\label{eq:length}
\qquad\qquad|\omega|_T=\sum_{i=1}^mt_i,\qquad \text{$\omega=(t_1,\dots, t_m)\in\Seq(T)$.}
\end{equation}      
Of particular historical interest is the case when $T\subseteq\mathbb{N}$, especially as it relates to the problem of counting \emph{restricted integer compositions}: for a fixed $n\in\mathbb{N}$, count the number of $\omega\in\Seq(T)$ satisfying $|\omega|_T=n$. Luckily this is a standard problem in analytic combinatorics \cite[\S I.3.1]{analytic-combinatorics} and can be attacked (asymptotically) via generating functions and an analysis of their associated poles and residues.       
\begin{theorem}\label{thm:generatingfunction} Let $T$ be a set of positive integers with at least two elements. For $n>0$ let 
\[a_{T,n}:=\#\{\omega\in\textup{Seq}(T)\,:\,|\omega|_T=n\}.\]
Then the ordinary generating function for the $a_{T,n}$ is \vspace{.1cm}
\[\mathcal{F}_T(z):=\sum_{n}a_{T,n}\,z^{n}=\frac{1}{1-(\,\sum_{t\in T}z^{t}\,)}.\]
Moreover, if $\gcd(T)=1$, then 
\[a_{T,n}=\frac{1}{\alpha\,G_T'(\alpha)}\,\alpha^{-n}\big(1+O(A^n)\big),\]
where $G_T:=\sum_{t\in T}z^{t}$, $0<\alpha<1$ is the unique solution to $G_T(\alpha)=1$, and $0<A<1$.    
\end{theorem} 
\begin{proof} See \cite[Theorem V.1]{analytic-combinatorics} and \cite[Example V.2]{analytic-combinatorics} on page 297 for a justification of why Theorem V.1 applies in this case.  
\end{proof}
\begin{remark}\label{rem:bounds} Here $\gcd(T)=1$ means that not all elements in $T\subseteq\mathbb{N}$ are a multiple of some common divisor $d>1$. In particular if $\gcd(T)=1$, if $\beta=\alpha^{-1}$, and if $0<\epsilon<1$, then \vspace{.15cm}  
\[
\frac{(1-\epsilon)\beta}{(\beta-1)\,G_T'(\frac{1}{\beta})}\beta^{R}-O_\epsilon(1)\leq\#\{\omega\in\Seq(T)\,:\,|\omega|_T\leq R\}\leq\frac{(1+\epsilon)\beta^3}{(\beta-1)\,G_T'(\frac{1}{\beta})}\beta^{R}+O_\epsilon(1) \vspace{.15cm}  \]
for all $R$ sufficiently large; here we sum the expression for $a_{T,n}$ in Theorem \ref{thm:generatingfunction} over $n\leq\lceil R\rceil$ (for the upper bound) and over $n\leq\lfloor R\rfloor$ (for the lower bound).  
\end{remark}  
Next, we prove a technical lemma that allows us to choose suitable rational approximations for each $\log(d_{\phi})$ and $\phi\in S$ (in order to use Theorem \ref{thm:generatingfunction} above); compare to \cite[Lemma 2.6]{Me:monoid}. 
\begin{lemma}\label{lem:diophantineapprox} Let $T$ be a positive set of real numbers with at least $2$ elements and let $\delta>0$. Then for all $t\in T$ there are integers $n_{t,\delta},m_{t,\delta},u_\delta\in\mathbb{N}$ depending of $\delta$ such that: \vspace{.1cm} 
\begin{enumerate} 
\item[\textup{(1)}] $t-\delta<\frac{n_{t,\delta}}{u_\delta}<t<\frac{m_{t,\delta}}{u_\delta}<t+\delta$.\vspace{.1cm}  
\item[\textup{(2)}] $\gcd(T_{1,\delta})=1=\gcd(T_{2,\delta})$ where $T_{1,\delta}=\{n_{t,\delta}:t\in T\}$ and $T_{2,\delta}=\{m_{t,\delta}:t\in T\}$.   \vspace{.1cm} 
\end{enumerate} 
Moreover if $T$ is uniformly discrete, then we can assume $n_{t,\delta}\neq n_{t',\delta}$ and $m_{t,\delta}\neq m_{t',\delta}$ for all distinct $t,t'\in T$ by choosing $\delta$ sufficiently small.    
\end{lemma} 
\begin{remark}\label{rem:delta} Recall that a set of real numbers is uniformly discrete if there exists a positive constant $\delta_T$ such that $|t-t'|>\delta_T$ for all distinct $t,t'\in T$. 
\end{remark} 
\begin{proof} Clearly integers $n_{t,\delta},m_{t,\delta},u_\delta\in\mathbb{N}$ satisfying condition (1) exists for all $\delta>0$. Therefore, to find integers satisfying both (1) and (2), we choose integers satisfying (1) and deform them to ensure that both conditions hold. Specifically, fix an integer $r>0$ and two arbitrary elements $t_1,t_2\in T$. Then define $v=(n_{t_2,\delta}\cdot m_{t_2,\delta})^r$ and \vspace{.1cm} 
\[n_{t_1,\delta}'= n_{t_1,\delta}\cdot v+1,\;\;\;\;n_{t,\delta}'= n_{t,\delta}\cdot v,\;\;\;m_{t_1,\delta}'= m_{t_1,\delta}\cdot v+1,\;\;\;\;m_{t,\delta}'= m_{t,\delta}\cdot v,\;\;\;\; u_\delta'=u_{\delta}\cdot v \vspace{.1cm} 
\]
for all $t\neq t_1$. Note that $m_{t_2,\delta}>n_{t_2,\delta}$ by condition (1) so that $v\geq2$. Now, we check that \vspace{.1cm} 
\[\frac{n_{t_1,\delta}'}{u_\delta'}=\frac{n_{t_1,\delta}}{u_\delta}+\frac{1}{u_\delta v},\;\;\;\;\frac{m_{t_1,\delta}'}{u_\delta'}=\frac{m_{t_1,\delta}}{u_\delta}+\frac{1}{u_\delta v},\;\;\;\; \frac{n_{t,\delta}'}{u_\delta'}=\frac{n_{t,\delta}}{u_\delta},\;\;\;\;\; \frac{m_{t,\delta}'}{u_\delta'}=\frac{m_{t,\delta}}{u_\delta} \vspace{.1cm}\]  
for all $t\neq t_1$. In particular, the new integers $n_{t,\delta}',m_{t,\delta}',u_\delta'\in\mathbb{N}$ also satisfy condition (1) for all $r$ sufficiently large. On the other hand, we will show that $\gcd(n_{t_1,\delta}',n_{t_2,\delta}')=1$ and $\gcd(m_{t_1,\delta}',m_{t_2,\delta}')=1$. To see this, suppose $p$ is a prime such that $p\big\vert n_{t_1,\delta}'$ and $p\big\vert n_{t_2,\delta}'$. Then since $p\big\vert n_{t_2,\delta}'$, either $p\big\vert n_{t_2,\delta}$ or $p\big\vert v$. But if $p\big\vert n_{t_2,\delta}$, then $p\big\vert v$ also by definition of $v$. Hence, in either case $p\big\vert v$. Therefore, both $p\big\vert n_{t_1,\delta}\cdot v$ and $p\big\vert n_{t_1,\delta}'$. However, it then follows that $p\big\vert (n_{t_1,\delta}'-n_{t_1,\delta}\cdot v)=1$, a contradiction. In particular, we deduce that $\gcd(n_{t_1,\delta}',n_{t_2,\delta}')=1$ as claimed. The same argument shows that $\gcd(m_{t_1,\delta}',m_{t_2,\delta}')=1$. In particular, the new integers $n_{t,\delta}',m_{t,\delta}',u_\delta'\in\mathbb{N}$ satisfy both conditions (1) and (2) for all $r$ sufficiently large. Moreover, it is straightforward to check that $n_{t,\delta}\neq n_{t',\delta}$ and $m_{t,\delta}\neq m_{t',\delta}$ for all $0<\delta<\delta_T$ by condition (1) alone; here $\delta_T$ is the constant from Remark \ref{rem:delta}, guaranteed to exist since $T$ is uniformly discrete.                
\end{proof} 
With these preliminaries in place, we are now ready to prove our first main result, a count on the number of functions $f\in M_S$ yielding a bounded height relation $H_\eta(f(P))\leq B$ when $M_S$ is free. This is statement (2) of Theorem \ref{thm:functioncount} from the introduction.  
\begin{theorem}\label{thm:functioncount'} Let $K$ be a global field, let $V$ be a projective variety over $K$, and let $S$ be a height controlled and uniformly log-discrete set of endomorphisms on $V$ with respect to a divisor class $\eta\in\Pic(V)\otimes\mathbb{R}$. If $M_S$ is free (non-abelian), then for all $\epsilon>0$ there is a positive constant $b=b(S,\eta,\epsilon)$ and a constant $B_{S,\eta}$ such that \vspace{.15cm} 
\begin{equation*}
(\log B)^{b}\ll\#\big\{f\in M_S\,:\, H_\eta(f(P))\leq B\big\}\ll (\log B)^{b+\epsilon}\vspace{.15cm}
\end{equation*} 
for all $P\in V$ with $H_\eta(P)>B_{S,\eta}$. Here the implicit constants depend on $P$, $S$, $\eta$ and $\epsilon$. 
\end{theorem} 
\begin{proof} Assume that $S$ is height controlled and uniformly log-discrete (with respect to a divisor class $\eta$) and that $M_S$ is free. In particular, if 
\[T=\{\log\deg_\eta(\phi):\phi\in S\}, \vspace{.1cm}\] 
then we have a bijection of sets $\chi: M_S\rightarrow \Seq(T)$ given by \vspace{.1cm} 
\begin{equation}\label{identification1} 
\;\;\chi(\theta_1\circ\dots\circ\theta_m)=\big(\log\deg_\eta(\theta_1),\dots,\log\deg_\eta(\theta_m)\big),\qquad\theta_i\in S. \vspace{.1cm} 
\end{equation} 
Moreover, since $\deg_\eta(F\circ G)=\deg_\eta(F)\deg_\eta(G)$ for all $F,G\in M_S$ by definition of $\deg_\eta(\cdot)$, we see from the identification in \eqref{identification1} that $\log\deg_\eta(f)=|\chi(f)|_T$ for all $f\in M_S$; here $|\cdot|_T$ is defined as in \eqref{eq:length}. Therefore,  \vspace{.1cm} 
\begin{equation}\label{identification2}
\#\{f\in M_S\,:\,\log\deg_\eta(f)\leq B\}=\#\{\omega\in \Seq(T)\,:\,|\omega|_T\leq B\} \vspace{.1cm} 
\end{equation}
holds for all $B$. Now for all $\delta>0$ and $t\in T$, we may choose $n_{t,\delta},m_{t,\delta}, u_\delta\in\mathbb{N}$ as in Lemma \ref{lem:diophantineapprox}. Moreover, we may assume that $\delta>0$ is sufficiently small to ensure that $n_{t,\delta}\neq n_{t',\delta}$ and $m_{t,\delta}\neq m_{t',\delta}$ for all distinct $t,t'\in T$. In particular, the map $\chi_1:\Seq(T)\rightarrow\Seq(T_{1,\delta})$ given by \[\chi_1\big((t_1,\dots,t_r)\big)=(n_{t_1,\delta}\,,\,\dots,\,n_{t_r,\delta})\] 
is injective (in-fact, bijective by definition). Moreover, part (1) of Lemma \ref{lem:diophantineapprox} implies that \vspace{.1cm} 
\[\qquad\frac{|\chi_1(\omega)|_{T_{1,\delta}}}{u_{\delta}}\leq |\omega|_{T}\qquad\text{for all $\omega\in \Seq(T)$.}\] 
Therefore, we deduce that \vspace{.1cm}   
\begin{equation}\label{identification3} 
\#\{\omega\in \Seq(T)\,:\,|\omega|_T\leq B\}\leq \#\{\omega_1\in \Seq(T_{1,\delta})\,:\,|\omega_1|_{T_{1,\delta}}\leq u_{\delta}B\}.  \vspace{.1cm}   
\end{equation} 
Similarly, the map $\chi_2:\Seq(T_{2,\delta})\rightarrow\Seq(T)$ given by  \vspace{.1cm} 
\[\chi_2\big((m_{t_1,\delta}\,,\,\dots,\,m_{t_r,\delta})\big)=(t_1,\dots,t_r)  \vspace{.1cm} \] 
is well-defined and injective, and hence  \vspace{.1cm} 
\begin{equation}\label{identification4} 
\#\{\omega_2\in \Seq(T_{2,\delta})\,:\,|\omega_2|_{T_{2,\delta}}\leq u_{\delta}B\}\leq\#\{\omega\in \Seq(T)\,:\,|\omega|_T\leq B\}  \vspace{.1cm} 
\end{equation} 
follows from part (1) of Lemma \ref{lem:diophantineapprox}. Therefore, we deduce from \eqref{identification2}, \eqref{identification3} and \eqref{identification4} that \vspace{.15cm} 
\begin{equation}\label{identification5}
\begin{split}
\#\{\omega_2\in \Seq(T_{2,\delta})\,:\,|\omega_2|_{T_{2,\delta}}\leq u_{\delta}B\}&\leq\#\{f\in M_S\,:\,\log\deg_\eta(f)\leq B\}\\[5pt]
&\leq\#\{\omega_1\in \Seq(T_{1,\delta})\,:\,|\omega_1|_{T_{1,\delta}}\leq u_{\delta}B\}. 
\end{split} 
\end{equation}  
\vspace{-.05cm} 

\noindent However, the sets $T_{i,\delta}\subseteq\mathbb{N}$ and so we can count the right and left sides of \eqref{identification5} using generating functions. To make this precise let $G_{i,\delta}(z)=G_{T_{i,\delta}}(z)=\sum_{t\in T_{i,\delta}}z^t$, let $\beta_{i,\delta}$ satisfy $G_{T_{i,\delta}}(\beta_{i,\delta}^{-1})=1$, and let $0<\epsilon'<1$. Then Theorem \ref{thm:generatingfunction} and Remark \ref{rem:bounds}, together with \eqref{identification5} imply \vspace{.15cm} 
\begin{equation}\label{identification6}
\begin{split} 
\Scale[1.1]{\frac{(1-\epsilon')\beta_{2,\delta}}{(\beta_{2,\delta}-1)\,G_{T_{2,\delta}}'\big(\frac{1}{\beta_{2,\delta}}\big)}\,\beta_{2,\delta}^{\;(u_{\delta}B)}-O_{\epsilon'}(1)}&\leq\#\{f\in M_S\,:\,\log\deg_\eta(f)\leq B\}\\[5pt]
&\leq\Scale[1.1]{\frac{(1+\epsilon')\beta_{1,\delta}^3}{(\beta_{1,\delta}-1)\,G_{T_{1,\delta}}'\big(\frac{1}{\beta_{1,\delta}}\big)}\,\beta_{1,\delta}^{\;(u_{\delta}B)}+O_{\epsilon'}(1)} 
\end{split} 
\end{equation} 
\vspace{-.05cm} 

\noindent holds for all $B$ sufficiently large (depending on $\delta$ and $\epsilon'$). On the other hand, a count on the number of functions of bounded log degree yields a count on the number of functions determining a bounded height relation (by Tate's Telescoping Lemma). Specifically, suppose that $P\in V(\overline{K})$ is such that $h_\eta(P)>b_{S,\eta}:=C_{S,\eta}/(d_{S,\eta}-1)$, where $C_{S,\eta}$ and $d_{S,\eta}$ are the constants from Definition \ref{def:htcontrolled,uniformdiscrete} above. Then, Tate's Telescoping Lemma \ref{lem:tate} implies that \vspace{.05cm}
\[\deg_\eta(f)(h_\eta(P)-b_{S,\eta})\leq h_\eta(f(P))\leq\deg_\eta(f)(h_\eta(P)+b_{S,\eta}).\vspace{.05cm}\] 
Therefore, for all $B$ we have the subset relations: \vspace{.15cm} 
\begin{equation}\label{degreestoheights}
\Scale[.87]{\begin{split} 
\bigg\{f\in M_S\,:\,\log\deg_\eta(f)\leq \log\bigg(\frac{B}{h_\eta(P)+b_{S,\eta}}\bigg)\bigg\}&\subseteq\big\{f\in M_S:\, h_\eta(f(P))\leq B\big\} \\[5pt] 
&\subseteq\bigg\{f\in M_S\,:\,\log\deg_\eta(f)\leq \log\bigg(\frac{B}{h_\eta(P)-b_{S,\eta}}\bigg)\bigg\}.\\[5pt]  
\end{split} }
\end{equation} 
In particular, if we replace $B$ with $\log(B/(h_\eta(P)+B_{S,\eta}))$ on the left side of \eqref{identification6}, replace $B$ with $\log(B/(h_\eta(P)-B_{S,\eta}))$ on the right side of \eqref{identification6}, and apply the change of base formulas for logarithms, then we deduce from \eqref{identification6} and \eqref{degreestoheights} that \vspace{.2cm}     
\begin{equation*} 
\begin{split} 
\frac{(1-\epsilon')\beta_{2,\delta}}{(\beta_{2,\delta}-1)\,G_{T_{2,\delta}}'\big(\frac{1}{\beta_{2,\delta}}\big)}\Bigg(\frac{B}{h_\eta(P)+b_{S,\eta}}&\Bigg)^{u_\delta\log(\beta_{2,\delta})}\hspace{-.7cm}-\;\;\;O_{\epsilon',\delta}(1)\\[8pt]
&\leq\#\big\{f\in M_S:\, h_\eta(f(P))\leq B\big\}\\[10pt] 
&\leq\frac{(1+\epsilon')\beta_{1,\delta}^3}{(\beta_{1,\delta}-1)\,G_{T_{1,\delta}}'\big(\frac{1}{\beta_{1,\delta}}\big)}\Bigg(\frac{B}{h_\eta(P)-b_{S,\eta}}\Bigg)^{u_\delta\log(\beta_{1,\delta})}\hspace{-.95cm}+\;\;\,O_{\epsilon',\delta}(1) 
\end{split}  
\end{equation*}
\vspace{-.02cm} 

\noindent Moreover, since point counting on varieties is normally phrased in terms of multiplicative heights, we replace $B$ with $\log B$ in the bounds directly above to deduce that \vspace{.15cm}    
\begin{equation}\label{bd:finalshape}
\Scale[.905]{C_{(2,P,\epsilon',\delta)}\log(B)^{b_{2,\delta}}-O_{\epsilon',\delta}(1)\leq\#\{f\in M_S\,:\, H_\eta(f(P))\leq B\}\leq C_{(1,P,\epsilon',\delta)}\log(B)^{b_{1,\delta}}+O_{\epsilon',\delta}(1)}.
\end{equation} 
\vspace{-.05cm} 

\noindent Here the explicit constants $C_{(i,P,\epsilon',\delta)}$ and the inexplicit $O_{\epsilon',\delta}(1)$ both depend on $S$ and $\eta$ as well, though we leave this dependence off in the notation. Specifically, \vspace{.25cm} 
\begin{equation}\label{bd:explicitconstants}
\Scale[1.125]{
\begin{matrix}
\;\;C_{(2,P,\epsilon',\delta)}=\frac{(1-\epsilon')\beta_{2,\delta}}{(\beta_{2,\delta}-1)\,G_{T_{2,\delta}}'\big(\frac{1}{\beta_{2,\delta}}\big)\big(h_\eta(P)+b_{S,\eta}\big)^{u_\delta\log(\beta_{2,\delta})}},&& b_{2,\delta}=u_\delta\log(\beta_{2,\delta}),\\
&&&\\
&&&\\
\;\;C_{(1,P,\epsilon',\delta)}=\frac{(1+\epsilon')\beta_{1,\delta}^3}{(\beta_{1,\delta}-1)\,G_{T_{1,\delta}}'\big(\frac{1}{\beta_{1,\delta}}\big)\big(h_\eta(P)-b_{S,\eta}\big)^{u_\delta\log(\beta_{1,\delta})}},&&b_{1,\delta}=u_\delta\log(\beta_{1,\delta}).
\\
&&&  
\end{matrix}}
\end{equation}
In particular \textbf{by fixing any $0<\epsilon'<1$}, we obtain the desired shape of the bound on heights in orbits in Theorem \ref{thm:functioncount'} from \eqref{bd:finalshape}.
It therefore remains to show that the positive difference $b_{1,\delta}-b_{2,\delta}$ can be made arbitrarily small (i.e., $<\epsilon$) by letting $\delta$ tend to zero. To see this we need the following convenient lower bound on $\alpha_{1,\delta}$; compare to \cite[Equation (26)]{Me:monoid}.
  
\begin{lemma}\label{rootlbd} Let $T$ and $T_{1,\delta}$ be as above and let $t_1:=\min(T)$. Then there exist positive constants $\delta_T$ and $c_T$ (depending only on $T$) such that ${c_T}^{\frac{1}{n_{t_1,\delta}}}\leq\alpha_{1,\delta}$ for all $0\leq\delta\leq\delta_{T}$.  \vspace{.1cm}     
\end{lemma} 

\begin{proof} We begin with a few observations about $T$. Since $T$ is uniformly discrete, we may choose a constant $\delta_T>0$ such that $|t-t'|>\delta_T$ for all distinct $t,t'\in T$. Hence for all $0<\delta<\delta_T$, it follows from Lemma \ref{lem:diophantineapprox} condition (1) that $n_{t,\delta}\neq n_{t'\delta}$ for all $t\neq t'$. Therefore, the association $t\rightarrow n_{t,\delta}$ is a bijection of sets $T\rightarrow T_{1,\delta}$. In particular, $T$ is countable. Moreover, $T$ must have a unique minimum element. To see this, note that for any fixed $t\in T$ there can be at most finitely many elements of $T$ less than $t$. Explicitly, we have that $\#\{t'\in T: t'< t\}\leq\lceil\frac{t}{\delta_T}\rceil$ by definition of $\delta_T$ and the fact that every element of $T$ is positive. Putting these facts together, we may enumerate $T=\{t_1, t_2, \dots\} $ and assume that $t_1<t_2<\dots$, i.e., the $t$'s are arranged in increasing order. 
Now assume that $\delta<\delta_T$. Then Lemma \ref{lem:diophantineapprox} condition (1) implies that    
\[\frac{t-\delta_T}{t_1}<\frac{t-\delta}{t_1}<\frac{n_{t,\delta}}{n_{t_1,\delta}}\]
for all $t\in T$. Moreover writing $t_m-\delta_T$ as a telescoping sum, we see that \vspace{.1cm}  
\[t_m-\delta_T=(t_m-t_{m-1})+(t_{m-1}-t_{m-2})+\dots+(t_2-t_1)+(t_1-\delta_T)>(m-1)\delta_T+(t_1-\delta_T). \vspace{.1cm}\] 
In particular, if we define the constant $e_T:=\delta_T/t_1$ (independent of $\delta$), then the bounds above together imply that \vspace{.1cm} 
\[(m-2)e_T+1<\frac{t_m-\delta_T}{t_1}<\frac{n_{t_m,\delta}}{n_{t_1,\delta}}\vspace{.1cm}\]
for all indices $m>1$ (this bound remains true for $m=1$, but we do not use it in that case). With this bound in mind, consider the auxiliary function \[\mathfrak{g}_T(z)=z+\sum_{m\geq2}^{\infty}z^{e_T (m-2)+1}=z+z\sum_{i=0}^{\infty}(z^{e_T})^{i}.\] 
In particular, it is clear that $\mathfrak{g}_T$ is convergent on $[0,1)$; in fact, $\mathfrak{g}_T=z+z/(1-z^{e_T})$ on this interval. With these preliminaries in place we reach the key fact:                  
\begin{equation}\label{bdwithG}
\,1=\sum_{n_{t,\delta}\in T_{1,\delta}}\hspace{-.2cm}(\alpha_{1,\delta})^{n_{t,\delta}}=\sum_{t\in T}{(\alpha_{1,\delta})}^{n_{t,\delta}}=(\alpha_{1,\delta})^{n_{t_1,\delta}}+\sum_{\substack{t\in T\\t\neq t_1}}((\alpha_{1,\delta})^{n_{t_1,\delta}})^{\frac{n_{t,\delta}}{n_{t_1,\delta}}}\leq \mathfrak{g}_T({\alpha_{1,\delta}}^{n_{t_1,\delta}}), 
\end{equation}
since ${(\alpha_{1,\delta})}^{n_{t_1,\delta}}\in(0,1]$ and smaller exponents yield larger values on this interval.
 On the other hand, $\mathfrak{g}_T$ is continuous, $\mathfrak{g}_T(0)=0$ and $1\leq \mathfrak{g}_T({\alpha_{1,\delta}}^{n_{t_1,\delta}})$ by \eqref{bdwithG}. Therefore, we may choose a constant $c_T\in(0,1)$ such that $\mathfrak{g}_T(c_T)=1/2$ by the Intermediate Value Theorem. Most importantly since $\mathfrak{g}_T$ is an increasing function we must have that $c_T\leq (\alpha_{1,\delta})^{n_{t_1,\delta}}$, independent of $\delta$ as claimed.      
\end{proof} 
From this point, we can complete the proof of Theorem \ref{thm:functioncount} by following the proof of \cite[Theorem 1.1]{Me:monoid}. To do this, we use the Mean Value Theorem applied to the functions $G_{2,\delta}(z)=\sum_{t\in T}z^{m_{t,\delta}}$ and $\mathfrak{h}_\delta(z)=u_\delta\log(z)$ on the intervals $[\alpha_{1,\delta},\alpha_{2,\delta}]$. Here and for the remainder of the proof, we assume that $0<\delta<\min\{\delta_T,t_1\}$; recall that $t_1:=\min(T)$. In particular, the sets $T$, $T_{1,\delta}$ and $T_{2,\delta}$ are all in bijection, and so we can use any of them for indexing sets for the countable, absolutely convergent sums below. We begin with a few estimates, all of which follow easily from part (1) of Lemma \ref{lem:diophantineapprox}.  
\begin{equation}\label{smallA}
\frac{2\delta}{t_1}<\frac{2\delta u_\delta}{n_{t_1,\delta}}<\frac{2\delta}{t_1-\delta},\qquad 1<\frac{m_{t_1,\delta}}{n_{t_1,\delta}}<\frac{t_1+\delta}{t_1-\delta},\qquad \frac{u_\delta}{m_{t_1,\delta}}<\frac{1}{t_1}.
\end{equation} 
Likewise, we have the lower bound on $\alpha_{1,\delta}$ from Lemma \ref{rootlbd}: 
\begin{equation}\label{smallB}
{c_T}^{\frac{1}{n_{t_1,\delta}}}\leq\alpha_{1,\delta}. 
\end{equation}
In particular, \eqref{smallA} and \eqref{smallB} together imply the following lower bound on the derivative:
\begin{equation}\label{smallC}
\begin{split} 
\Scale[.98]{G_{2,\delta}'(\alpha_{1,\delta})=\sum_{t\in T}m_{t,\delta}(\alpha_{1,\delta})^{m_{t,\delta}-1}}&\geq m_{t_1,\delta}(\alpha_{1,\delta})^{m_{t_1,\delta}-1} \\ 
&\Scale[.98]{\geq m_{t_1,\delta}(\alpha_{1,\delta})^{m_{t_1,\delta}}\geq m_{t_1,\delta}\,{c_T}^{\frac{m_{t_1,\delta}}{n_{t_1,\delta}}}\geq m_{t_1,\delta}\,{c_T}^{\frac{t_1+\delta}{t_1-\delta}}}.
\end{split}  
\end{equation}                      
Similarly, \eqref{smallA} and \eqref{smallB} together imply that:
\begin{equation*}
\begin{split} 
G_{2,\delta}(\alpha_{1,\delta})=\sum_{t\in T}{\alpha_{1,\delta}}^{m_{t,\delta}}&=\sum_{t\in T}{\alpha_{1,\delta}}^{(\frac{m_{t,\delta}}{u_\delta}-\frac{n_{t,\delta}}{u_\delta})u_\delta}\cdot{\alpha_{1,\delta}}^{n_{t,\delta}}\\[3pt]
&\geq \sum_{t\in T}{\alpha_{1,\delta}}^{2\delta u_\delta}\cdot{\alpha_{1,\delta}}^{n_{t,\delta}}\\[3pt]
&={\alpha_{1,\delta}}^{2\delta u_\delta}\cdot\sum_{t\in T}{\alpha_{1,\delta}}^{n_{t,\delta}}\\[3pt]
&={\alpha_{1,\delta}}^{2\delta u_\delta}\geq {c_T}^{\frac{2\delta u_\delta}{n_{t_1,\delta}}}\geq {c_T}^{\frac{2\delta}{t_1-\delta}}.  
\end{split} 
\end{equation*} 
Here, we use also that $0\leq \frac{m_{t,\delta}}{u_\delta}-\frac{n_{t,\delta}}{u_\delta}\leq 2\delta$ by construction; see Lemma \ref{lem:diophantineapprox} part (1). In particular, we deduce the following key upper bound:  
\begin{equation}\label{smallD}
G_{2,\delta}(\alpha_{2,\delta})-G_{2,\delta}(\alpha_{1,\delta})=1-G_{2,\delta}(\alpha_{1,\delta})\leq 1-{c_T}^{\frac{2\delta}{t_1-\delta}}. \vspace{.1cm} 
\end{equation}
We are now ready to apply the Mean Value Theorem to $G_{2,\delta}(x)$ on $[\alpha_{1,\delta},\alpha_{2,\delta}]$. Specifically, \vspace{.1cm}
\begin{equation*}
m_{t_1,\delta}\,{c_T}^{\frac{t_1+\delta}{t_1-\delta}}\leq G_{2,\delta}'(\alpha_{1,\delta})=\min_{\alpha_{1,\delta}\leq x\leq \alpha_{2,\delta}}G_{2,\delta}'(x)\leq \frac{G_{2,\delta}(\alpha_{2,\delta})-G_{2,\delta}(\alpha_{1,\delta})}{\alpha_{2,\delta}-\alpha_{1,\delta}}\leq\frac{1-{c_T}^{\frac{2\delta}{t_1-\delta}}}{\alpha_{2,\delta}-\alpha_{1,\delta}} \vspace{.1cm} 
\end{equation*}
follows from \eqref{smallC}, \eqref{smallD}, and the Mean Value Theorem. Therefore, we have the estimate:  
\begin{equation}\label{smallE}
0\leq \alpha_{2,\delta}-\alpha_{1,\delta}\leq \frac{1-{c_T}^{\frac{2\delta}{t_1-\delta}}}{m_{t_1,\delta}\,{c_T}^{\frac{t_1+\delta}{t_1-\delta}}}\;.
\end{equation} 
Likewise, the Mean Value Theorem for $\mathfrak{h}_\delta(x)=u_\delta\log(x)$ on $[\alpha_{1,\delta},\alpha_{2,\delta}]$, \eqref{smallB}, and the fact that $n_{1,\delta}>0$ together yield \vspace{.1cm}  
\begin{equation}\label{MVTLog}
0\leq\frac{\mathfrak{h}_\delta(\alpha_{2,\delta})-\mathfrak{h}_\delta(\alpha_{1,\delta})}{\alpha_{2,\delta}-\alpha_{1,\delta}}\leq\max_{\alpha_{1,\delta}\leq x\leq\alpha_{2,\delta}}\mathfrak{h}_\delta'(x)=\mathfrak{h}_\delta'(\alpha_{1,\delta})=u_\delta(\alpha_{1,\delta})^{-1}\leq \frac{u_\delta}{c_T}. \vspace{.1cm} 
\end{equation} 
Hence, after combining \eqref{bd:explicitconstants},\eqref{smallA}, \eqref{smallE} and \eqref{MVTLog}, we deduce that \vspace{.15cm} 
\begin{equation}\label{smallF}
\begin{split}
0\leq b_{1,\delta}-b_{2,\delta}=\mathfrak{h}_\delta(\alpha_{2,\delta})-\mathfrak{h}_\delta(\alpha_{1,\delta})&\leq  \frac{u_\delta}{c_T}\cdot \frac{1-(c_T)^{\frac{2\delta}{t_1-\delta}}}{m_{t_1,\delta}(c_T)^{\frac{t_1+\delta}{t_1-\delta}}} \\[4pt]
&= \frac{1}{c_T}\cdot\frac{u_\delta}{m_{t_1,\delta}}\cdot \frac{1-(c_T)^{\frac{2\delta}{t_1-\delta}}}{(c_T)^{\frac{t_1+\delta}{t_1-\delta}}}\leq\frac{1}{c_T\,t_1}\cdot\frac{1-(c_T)^{\frac{2\delta}{t_1-\delta}}}{(c_T)^{\frac{t_1+\delta}{t_1-\delta}}} 
\end{split} 
\end{equation} 

\vspace{.15cm} 
\noindent However, the upper bound in \eqref{smallF} goes to zero as $\delta$ goes to zero (since $c_T>0$). Therefore, the exponents $b_{1,\delta}$ and $b_{2,\delta}$ in \eqref{bd:finalshape} can be made arbitrarily close as claimed.
\end{proof} 
\begin{remark} \label{rem:explicit}
To approximate $b_{1,\delta}$ and $b_{2,\delta}$ in practice, we first approximate the solutions to $G_{i,\delta}(z)=1$ by similar solutions for associated cutoff functions (which are rational). To wit, let $N>0$ be any positive integer and define the rational functions 
\[G_{1,\delta,N}:=\sum_{n\leq N\atop n\in T_{1,\delta}}\mathclap{z^n}\;+\frac{z^N}{1-z}\qquad\text{and}\qquad G_{2,\delta,N}:=\sum_{m\leq N\atop m\in T_{2,\delta}}\mathclap{z^m}\]
associated to $G_{1,\delta}$ and $G_{2,\delta}$ respectively. Then since $n_t<m_t$ for all $t\in T$ by construction, we see that \vspace{.1cm}   
\begin{equation}\label{bd:G's}
G_{2,\delta,N}(z)\leq G_{2,\delta}(z)< G_{1,\delta}(z)\leq G_{1,\delta,N}(z)\qquad \text{for all\, $0<z<1$.} \vspace{.1cm} 
\end{equation}
Now let $\alpha_{i,\delta}\in(0,1)$ satisfy $G_{i,\delta}(\alpha_{i,\delta})=1$ and let $\alpha_{i,\delta,N}\in(0,1)$ satisfy $G_{i,\delta,N}(\alpha_{i,\delta,N})=1$. In particular, $\beta_{i,\delta}={(\alpha_{i,\delta})}^{-1}$ and $b_{i,\delta}=u_\delta\log(\beta_{i,\delta})$ by definition of $\beta_{i,\delta}$ and \eqref{bd:explicitconstants} above. Likewise, we define $\beta_{i,\delta,N}={(\alpha_{i,\delta,N})}^{-1}$ and $b_{i,\delta,N}=u_\delta\log(\beta_{i,\delta,N})$. Then \eqref{bd:G's} implies that  \vspace{.1cm}     
\begin{equation}\label{bd:beta's}
\begin{split} 
\alpha_{1,\delta,N}&\leq\,\alpha_{1,\delta}<\,\alpha_{2,\delta}\leq\,\alpha_{2,\delta,N}\\[6pt] 
\beta_{2,\delta,N}&\leq\, \beta_{2,\delta}<\,\beta_{1,\delta}\leq\,\beta_{1,\delta,N}. \\[6pt] 
b_{2,\delta,N}&\leq\, b_{2,\delta}\,<\,b_{1,\delta}\,\leq\,b_{1,\delta,N}. 
\end{split}
\end{equation}

\vspace{.1cm} 
\noindent Here we use also that both the $G_{i,\delta}$ and the $G_{i,\delta,N}$ are strictly increasing on $(0,1)$. Moreover, $\alpha_{1,\delta,N}$ and $\alpha_{2,\delta,N}$ are algebraic, and we can (in practice) estimate these quantities using known root finding algorithms (and {\tt{Magma}} \cite{Magma}) to any degree of accuracy. In this way, we can give explicit bounds for the growth rate of $\#\{f\in M_S\,:\, H_\eta(f(P))\leq B\}$ in practice. We use this in section \ref{sec:free} to give explicit bounds for the number of points of bounded height in semigroup orbits generated by the polynomials $z^q+1$, where $q$ is in various subsets of the prime numbers (e.g., the Mersenne primes); see, for instance, Corollary \ref{cor:amusing} from the introduction.      
\end{remark}   
Lastly, we can use the bounds in Theorem \ref{thm:functioncount'} on the number of functions in \emph{free} semigroups satisfying a bounded height relation to give an upper bound on the number of points of bounded height in arbitrary semigroup orbits. This is part (1) of Theorem \ref{thm:functioncount} from the introduction. 
\begin{corollary}\label{cor:upper} Let $K$ be a global field, let $V$ be a projective variety over $K$, and let $S$ be a height controlled and uniformly log-discrete set of endomorphisms on $V$ with respect to a divisor class $\eta\in\Pic(V)\otimes\mathbb{R}$. Then there is a positive constant $b=b(S,\eta)$ and a constant $B_{S,\eta}$ such that 
\begin{equation*}
\#\big\{Q\in \Orb_{S}(P)\,:\, H_\eta(Q)\leq B\big\}\ll (\log B)^{b}\vspace{.15cm}
\end{equation*} 
for all $P\in V$ with $H_\eta(P)>B_{S,\eta}$.  
\end{corollary} 
\begin{proof} Suppose $S$ is a height controlled and uniformly log-discrete set of endomorphisms on $V$. We may assume that $S$ has at least two elements:  otherwise the desired bound, in fact with a power of $\log(B)$ replaced with $\log(\log(B))$, follows easily from the celebrated results of Call and Silverman \cite{Call-Silverman}. Now let $F_S$ be the free semigroup generated by $S$ under concatenation. Then, given a word $w=\theta_1\dots\theta_n\in F_S$, we can define an action of $w$ on $V$ via $w\cdot P= \theta_1\circ\dots\circ\theta_n(P)$. Likewise, we define the $\eta$-degree of $w$ to be $\deg_\eta(\theta_1\circ\dots\circ\theta_n)$. In particular, (by counting words of bounded degree) it is straightforward to see that we can replace $M_S$ with $F_S$ in the proof of Theorem \ref{thm:functioncount'}, choose any $0<\delta<\delta_T$, and deduce that  \vspace{.1cm}
\[\log(B)^{b_{2,\delta}}\ll\#\{w\in F_S\,:\, H_\eta(w\cdot P)\leq B\}\ll\log(B)^{b_{1,\delta}} \vspace{.1cm}\]  
for all $P$ with $H_\eta(P)> B_S$ as before; see \eqref{bd:finalshape}. In particular, since every point $Q\in\Orb_S(P)$ is of the form $Q=w\cdot P$ for some $w\in F_S$, we have that \vspace{.1cm}
\[\#\{Q\in\Orb_S(P)\,:\, H_\eta(Q)\leq B\}\leq\#\{w\in F_S\,:\, H_\eta(w\cdot P)\leq B\}\ll\log(B)^{b_{1,\delta}}. \vspace{.15cm}\]   
Therefore, the number of points in $\Orb_S(P)$ with height at most $B$ is $\ll\log(B)^{b_{1,\delta}}$.   
\end{proof}                 
 
\section{Free semigroups generated by unicritical polynomials}\label{sec:free}
In order to apply both the upper and lower bounds from Theorem \ref{thm:functioncount} to some explicit, infinitely generated examples, we need three properties to hold: the full semigroup is free, the underlying generating set is height controlled, and the set of degrees of the generating set is uniformly log-discrete. This last property is the easiest to ensure; if we can find an infinitely generated example that is free and height controlled, then we can restrict to a subset of the generating set on which all three properties hold. Likewise freeness, though it can be nontrivial to prove, should hold generically; see, for instance,  \cite[Proposition 4.1]{Me:monoid}. Therefore, the property of being height controlled is the most restrictive of the three (for infinite generating sets). Nevertheless, we can produce examples of infinite height controlled sets using unicritical polynomials as in \cite[Example 2.7]{Me:stochastic}. With this in mind, we first show that these polynomials generate free semigroups under composition. This fact is perhaps well known to the experts (at least in the finitely generated case), but without a suitable reference we include a proof here:       
      
\begin{theorem}\label{thm:free} Let $k$ be a field of characteristic zero and let $S=\big\{z^d+c:\,d\geq2,\, c\in k^*\big\}$. Then the semigroup generated by $S$ under composition is free.     
\end{theorem} 
As a first step, we note that there are no non-constant solutions to a Cassels-Catalan equation over the polynomial ring $k[z]$, a simple consequence of Mason's ABC theorem \cite{Mason}.    
\begin{lemma}\label{lem:catalan} Let $k$ be a field of characteristic zero, let $a,b,c\in k^*$, and let $m,n\geq2$. Then there are no non-constant solutions $x,y\in k[z]$ to the equation $ax^n+by^m=c$
\end{lemma} 
\begin{proof} Suppose that $x,y\in k[z]$ are non-constant solutions to the equation  $ax^n+by^m=c$. Note that since $a,b,c\in k^*$, the triple $(A,B,C)=(ax^n,by^m,c)$ is pairwise coprime in $k[z]$. In particular, Mason's ABC Theorem over $k[z]$ implies that  \vspace{.1cm}    
\begin{equation}
\begin{split}  
\max\big\{n\deg(x), m\deg(y)\big\}&=\max\{\deg(A),\deg(B),\deg(C)\}\\[5pt] 
&\leq\deg(\text{rad}(ABC))-1\leq\deg(x)+\deg(y)-1;
\end{split} 
\end{equation} 
here $\text{rad}(f)$ denotes the number of distinct roots of $f\in k[z]$. On the other hand $n,m\geq 2$, hence 
\[2\deg(x)\leq \deg(x)+\deg(y)-1\;\;\;\text{and}\;\;\;2\deg(y)\leq \deg(x)+\deg(y)-1. \vspace{.1cm}\] 
In particular, these inequalities together imply that 
\[\deg(x)\leq\deg(y)-1\leq\deg(x)-2,\]
a contradiction. Therefore, there are no non-constant solutions $x,y\in k[z]$ to the equation $ax^n-by^m=c$ as claimed.     
\end{proof} 
\begin{proof}[(Proof of Theorem \ref{thm:free})] Now suppose that we have a relation
\begin{equation}\label{relation}   
\theta_1\circ\dots\circ\theta_r=\tau_1\circ\dots\circ\tau_k
\end{equation} 
for some $\theta_i,\tau_j\in S$. We will show that $r=k$ and $\theta_i=\tau_i$ for all $1\leq i\leq r$. To do this, write $\theta_1=z^{d_1}+c_1$ and $\tau_1=z^{d_2}+c_2$. Clearly if $n=m=1$, then $\theta_1=\tau_1$  and there is nothing to prove. Therefore, we may assume without loss of generality that $r\geq k$ and $r\geq2$. Let $f_1=\theta_2\circ\dots\circ\theta_r$, let $g_1=\tau_2\circ\dots\circ\tau_k$ if $k\geq2$, and let $g_1$ be the identity if $k=1$ (all non-constant polynomials). Then \eqref{relation} implies that
\begin{equation}\label{relation2}
f_1^{d_1}-g_1^{d_2}=c_2-c_1.
\end{equation}  
In particular, $c_1=c_2$, since otherwise $(x,y)=(f_1(z),g_1(z))$ is a non-constant solution to the Cassels-Catalan equation given by $(a,b,c,n,m)=(1,-1,c_1-c_2,d_1,d_2)$, which contradicts Lemma \ref{lem:catalan}. Therefore, \eqref{relation} becomes $f_1^{d_1}=g_1^{d_2}$. Now suppose $d=\gcd(d_1,d_2)$ so that $d_1=e_1d$ and $d_2=e_2d$ for some coprime positive integers $e_1$ and $e_2$. Then 
\[(f_1^{e_1})^d=f_1^{d_1}=g_1^{d_2}=(g_1^{e_2})^d\] 
and thus  $(f_1^{e_1})/(g_1^{e_2})$ is a $d$-th root of unity (in particular, a constant function). Hence, we may write $f_1^{e_1}=\zeta_d\, g_1^{e_2}$ for some $d$-th root of unity $\zeta_d$. On the other hand $f_1,g_1\in M_S$, and so $f_1$ and $g_1$ are both monic polynomials. Likewise, $f_1^{e_1}$ and $g_1^{e_2}$ are monic. In particular, the relation  $f_1^{e_1}=\zeta_d\, g_1^{e_2}$ forces $\zeta_d=1$, and we deduce that $f_1^{e_1}=g_1^{e_2}$ for some coprime exponents $e_1$ and $e_2$. Therefore after writing $1=we_1+ve_2$ for some $u,v\in\mathbb{Z}$, we see that 
\[f_1=f_1^{we_1+ve_2}=(f_1^{e_1})^w\cdot (f_1^{v})^{e_2}=(g_1^{e_2})^w\cdot (f_1^{v})^{e_2}=(f_1^{v}\cdot g_1^{w})^{e_2}.\]
Likewise, it is easy to check that $g_1=(f_1^{v}\cdot g_1^{w})^{e_1}$. In particular, if we let $u=f_1^{v}\cdot g_1^{w}\in k(z)$, then we have shown that 
\begin{equation}\label{powers}
f_1=u^{e_2}\;\;\;\text{and}\;\;\; g_1=u^{e_1}.
\end{equation}  
Note also that $u$ must in fact lie in the polynomial ring, since it is integral over $k[z]$ and $k[z]$ is integrally closed. Now then suppose that either $e_1$ or $e_2$ is strictly bigger than one (we'll show that this is impossible). Suppose first that $e_2>1$. Then writing $f_1=f_2^{d_3}+c_3$ for some $f_2\in M_S$ (possibly the identity function), we see that \eqref{powers} implies that 
\[f_2^{d_3}+c_3=u^{e_2}.\]
However, this contradicts Lemma \ref{lem:catalan} since in that case $(x,y)=(f_2(z),u(z))$ is a non-constant solution to the Cassels-Catalan equation given by $(a,b,c,n,m)=(1,-1,-c_3,d_3,e_2)$; here we use crucially that $c_3\neq0$ by definition of $S$. Hence, $e_2$ must be equal to one. Repeating the same argument when $k\geq2$ with the relation $g_1=u^{e_1}$ from \eqref{powers}, we see that $e_1$ must equal one also. On the other hand if $k=1$, then $g_1=z$ and \eqref{powers} immediately implies that $e_1=1$ (since $z$ is irreducible). In particular $e_1=1=e_2$ in either case, and it follows by definition of $e_1$ and $e_2$ that $d_1=d=d_2$. Hence, coupled with the fact that we have already shown that $c_1=c_2$, we deduce that $\theta_1=\tau_1$. Therefore, \eqref{relation2} implies that $f_1^d=g_1^d$ so that $f_1=\zeta_d g_1$ for some root of unity $\zeta_d$. However, again $f_1$ and $g_1$ are both in $M_S$ and therefore monic. In particular, $\zeta_d=1$ and $f_1=g_1$. Note that this precludes the possibility that $k=1$, since otherwise $d_3=\deg(f_1)=\deg(g_1)=1$ by simply equating degrees.  

To recap, we have shown that if $\theta_1\circ\dots\circ\theta_r=\tau_1\circ\dots\circ\tau_k$ for some $\theta_i,\tau_j\in S$ and some $r\geq2$, then $k\geq2$, $\theta_1=\tau_1$, and $\theta_2\circ\dots\circ\theta_r=f_1=g_1=\tau_2\circ\dots\circ\tau_k$. Repeating this argument until the $\tau's$ are eliminated (possible since $r\geq k$), we see that $\theta_i=\tau_i$ for all $1\leq i\leq k$. On the other hand if $r>k$, then after equating the degrees in the original expression $\theta_1\circ\dots\circ\theta_r=\tau_1\circ\dots\circ\tau_k$, we see that 
$\deg(\theta_{k+1})\cdots\cdot\deg(\theta_r)=1$. But, this is impossible since each $\theta\in S$ has degree at least $2$. 	Hence, $r=k$ and the expression in \eqref{relation} is the trivial one. Therefore, $M_S$ is free as claimed.                        
\end{proof} 
\begin{remark} Note that if we include the constant terms $c=0$ in $S$, then $M_S$ will not be free (monic power maps commute). Likewise, if we allow non-monic unitcritical polynomials, then $M_S$ may fail to be free. For instance, if $\phi_1=z^2+1$ and $\phi_2=-\phi_1=-z^2-1$, then $\phi_1\circ\phi_1=\phi_1\circ\phi_2$ is a non-trivial compositional relation. In this way, Theorem \ref{thm:free} is in some sense the strongest possible statement.   
\end{remark}
In particular, we can put Theorem \ref{thm:free} together with \cite[Lemma 12]{Ingram} to produce infinitely generated semigroups for which the upper and lower bounds in Theorem \ref{thm:functioncount} hold. Recall that $\mathfrak{d}\subset\mathbb{N}$ is called \emph{uniformly log-discrete} if there is a constant $\delta$ such that $|\log d-\log d'|>\delta$ for all distinct $d,d'\in\mathfrak{d}$.    
\begin{corollary}\label{cor:examples} Let $\mathfrak{d}\subset\mathbb{N}$ be \emph{uniformly log-discrete}, and let $c\in\overline{\mathbb{Q}}^{*}$. Then 
\[S_{\mathfrak{d},c}:=\{z^d+c\,:\,d\in\mathfrak{d}, d\geq2\}\]
is height controlled and uniformly log-discrete. Moreover, the semigroup generated by $S_{\mathfrak{d},c}$ under composition is free. In particular, statement (2) of Theorem \ref{thm:functioncount} holds for $S_{\mathfrak{d},c}$.      
\end{corollary}
\begin{proof} It follows from \cite[Lemma 12]{Ingram} that 
\begin{equation}\label{eq:unicrit+}
|h(\phi(P))-\deg(\phi)h(P)|\leq h(c)+\log2
\end{equation} 
for all $P\in\mathbb{P}^1(\overline{\mathbb{Q}})$ and all $\phi\in S_{\mathfrak{d},c}$. Hence, $S_{\mathfrak{d},c}$ is height controlled. Likewise, $S_{\mathfrak{d},c}$ is uniformly log-discrete by our assumption on $\mathfrak{d}$. Finally, the semigroup generated by $S_{\mathfrak{d},c}$ under composition is free since $S_{\mathfrak{d},c}\subset S$ and $M_S$ is free by Theorem \ref{thm:free}.         
\end{proof} 
\begin{remark} In particular, $S_{a,b,c}:=\{x^d+c\,:\,\text{$d=a^n+b$ for some $n\geq1$}\}$ from Example \ref{eg:htcontrolledandlogdiscrete} above satisfy Corollary \ref{cor:examples} and Theorem \ref{thm:functioncount} part (2).
\end{remark}  
\begin{remark}Technically, if we include finitely many $c$'s (i.e., constant terms), then $S_{\mathfrak{d},c}$ is still height controlled. However, allowing more than one $c$ precludes $S_{\mathfrak{d},c}$ from being uniformly log-discrete.     
\end{remark} 

\section{Orbits generated by unicritical polynomials: the Catalan case}
Now that we know that polynomials of the form $z^d+c$ generate free semigroups, we would like to use Theorem \ref{thm:functioncount} to count points (not just functions) of bounded height in orbits. To do this for any given basepoint $P$, we need to control the number of distinct functions $f$ and $g$ in the relevant semigroup that can agree at $P$, i.e., $f(P)=g(P)$. Perhaps unsurprisingly, this is possible for finitely generated semigroups over $\mathbb{P}^1$ by invoking Siegel's integral point theorem and the polynomial classification theorem of Bilu-Tichy \cite{Bilu}; essentially, you need only control the integral points on the finitely many curves $f(x)=g(y)$ where $f$ and $g$ are in the generating set; see \cite{Me:monoid} for details. However, the problem is harder when one allows infinitely generated semigroups (which is in some sense the point of this paper), since there are infinitely many curves to control. On the other hand, there is a case where this is possible, namely when we restrict ourselves to polynomials of the form $z^q+1$ (with prime exponent and constant term one); to do this, we use known results for integral solutions to the Catalan equations.   

With this in mind, we set some notation used throughout this section. Fix a number field $K$ and let $w_K$ be the number of roots of unity in $K$. Then we let $S_K$ be the set of polynomials 
\begin{equation}\label{SK}
S_K:=\{z^q+1\,:\,\text{$q$ is prime and $\gcd(q,w_K)=1$} \}
\end{equation} 
with prime degree (coprime to $w_K$) and constant term one. Likewise, since we need our generating sets to be uniformly log-discrete to apply Theorem \ref{thm:functioncount}, we consider subsets of $S_K$ also. Namely, given a uniformly log-discrete set of primes $\mathfrak{q}$ (meaning there is a constant $\delta>0$ such that $|\log(q)-\log(q')|>\delta$ for all distinct $q,q'\in\mathfrak{q}$), we define 
\begin{equation}\label{SKq}
S_{K,\mathfrak{q}}:=\{z^q+1\,:\,\text{$q\in\mathfrak{q}$ and $\gcd(q,w_K)=1$}\}.
\end{equation} 
Then to pass from counting functions in $M_{S_K}$ and $M_{S_{K,\mathfrak{q}}}$ to counting points in orbits, we need the following result, which bounds the  solutions to the Catalan equation over suitable rings of integers. 
\begin{theorem}[Brindza \cite{Brindza}]\label{thm:catalanSintegers} Let $K$ be a number field, let $\mathfrak{s}$ be a finite set of places of $K$ containing the archimedean ones, and let $\mathfrak{o}_{K,\mathfrak{s}}$ be the ring of $\mathfrak{s}$-integers of $K$. Then there exists an effectively computable constant $\kappa_{\mathfrak{s}}>2$ such that if 
\begin{equation}\label{CatalanEquation}
x^n-y^m=1
\end{equation} 
for some $n,m\in\mathbb{N}$ with $nm>4$ and some $x,y\in\mathfrak{o}_{K,\mathfrak{s}}$ which are not roots of unity, then 
\[\max\{h(x),h(y),n,m\}\leq \kappa_\mathfrak{s}.\]    
\end{theorem} 
Now, we may use Brindza's bound on $\mathfrak{s}$-integral solutions to the Catalan equation to control the set of functions with common value at $P$ in the semigroups generated by \eqref{SK} and \eqref{SKq}. In particular, in this step we do not need to assume that the set of degrees form a uniformly log-discrete set; our argument works for the full semigroup generated by $z^q+1$ and $q$ prime, as long as $\gcd(q,w_K)=1$. In what follows, given $f=\theta_1\circ\dots\circ\theta_n\in M_{S_K}$ for some $\theta_i\in S_K$, we define the length of $f$ to be $\ell(f)=n$. Note that this is well-defined since $M_{S_K}$ is free.    
\begin{lemma}\label{lem:catalanbd} Let $S_K$ be as in \eqref{SK} and suppose that $P\in\mathbb{P}^1(K)$ satisfies $h(P)>\log4$. Then there exists a constant $d_{K,P}$ (depending only on $K$ and $P$) such that if 
\begin{equation}\label{eq:relation} 
\theta\circ f(P)=\tau\circ g(P)
\end{equation} 
for some $\theta,\tau\in S_K$ and some $f,g\in M_{S_K}$, then either \vspace{.1cm} 
\[\text{$\theta=\tau$ and $f(P)=g(P)$\qquad or \qquad $\max\big\{\deg(\theta\circ f),\,\deg(\tau\circ g)\big\}\leq\,d_{K,P}$}.\]     
\end{lemma} 
\begin{proof} Assume that \eqref{eq:relation} holds and choose a set of places $\mathfrak{s}$ so that $P$ is $\mathfrak{s}$-integral. Now define $d_{K,P}$ depending on the constant $\kappa_{\mathfrak{s}}$ from Lemma \ref{lem:catalanbd} as follows: 
\[d_{K,P}:=\max\bigg\{\frac{\kappa_{\mathfrak{s}}^2\;h(P)}{\log(2)\,h_{K}^{\min}},\;\frac{\kappa_{\mathfrak{s}}^3}{\log2}\bigg\};\]
here $h_{K}^{\min}$ is the minimum height of the elements in $K$ which are not roots of unity. To proceed, we argue in cases. In what follows, we write $\theta(z)=z^{d_1}+1$ and $\tau(z)=z^{d_2}+1$ for some primes $d_1$ and $d_2$ coprime to the number of roots of unity in $K$. \\[3pt] 
\textbf{Case (1):} Suppose that $\ell(f)=\ell(g)=0$, i.e., $f$ and $g$ are both the identity function. Then \eqref{eq:relation} implies that $P^{d_1}=P^{d_2}$, and therefore $P^{d_1-d_2}=1$. Hence if $d_1\neq d_2$, then $P$ is a root of unity. But this implies that $h(P)=0$, contradicting our assumption that $h(P)>\log4$. In particular, $\theta=\tau$ and $f=g$ in this case.        \\[4pt] 
\textbf{Case (2):} Suppose that either $\ell(f)=0$ or $\ell(g)=0$, but not both. Without loss assume that $\ell(f)\geq1$ and $\ell(g)=0$, so we may write $f=\omega\circ f_1$ for some $\omega(z)=z^{d_3}+1\in S_K$ and some $f_1\in M_{S,K}$. Then \eqref{eq:relation} implies that 
\begin{equation}\label{case2}
f(P)^{d_1}=P^{d_2}.
\end{equation} 
In particular if $\tau=\theta$, then $d_1=d_2$ and $f(P)/P$ is a $d_1$'st root of unity. But $d_1$ is coprime to $w_K$ (the number of roots of unity in $K$) and $f(P)/P\in K$, so that $f(P)=P$. On the other hand, the lower bound in Tate's Telescoping Lemma \ref{lem:tate} and \eqref{eq:unicrit+} imply that 
\[3(h(P)-\log2)\leq\deg(f)(h(P)-\log2)\leq h(f(P))=h(P).\]
However, this implies that $2h(P)\leq3\log2$, again violating our assumption that $h(P)>\log4$. In particular, it must be the case that $\tau\neq\theta$, and thus $d_1\neq d_2$ are distinct odd primes. Therefore, there are integers $a_1$ and $a_2$ such that $a_1d_1+a_2d_2=1$, from which it follows from \eqref{case2} that
\begin{equation}\label{case2-catalan}
f(P)=t^{d_2}\;\;\text{and}\;\;P=t^{d_1},\qquad\text{for}\;\;t=f(P)^{a_2}\cdot P^{a_1}.
\end{equation} 
Note that $t$ is $\mathfrak{s}$-integral since $P$ is $\mathfrak{s}$-integral and the ring of $\mathfrak{s}$-integers is integrally closed. On the other hand, $f=\omega\circ f_1=f_1^{d_3}+1$ so that the left side of \eqref{case2-catalan} implies that $(x,y,n,m)=(t,f_1(P),d_2,d_3)$ is an $\mathfrak{s}$-integral solution to the Catalan equation \eqref{CatalanEquation}. Moreover, $f_1(P)$ cannot be a root of unity, since $\log2\leq\deg(f_1)(h(P)-\log2)=h(f_1(P))$ by Tate's Telescoping Lemma and our assumption on the height of $P$. Likewise, $t$ cannot be a root of unity since then $P$ (which equals $t^{d_1}$) must be. Therefore, Theorem \ref{thm:catalanSintegers} implies that 
\begin{equation}\label{case2:bd} 
\max\{h(t),h(f_1(P)),d_2,d_3\}\leq \kappa_{\mathfrak{s}}.
\end{equation}  
However, again applying Tate's telescoping Lemma to $f_1$, we deduce from \eqref{case2:bd} that 
\[
\deg(f_1)\log2\leq\deg(f_1)(h(P)-\log2)=h(f_1(P))\leq\kappa_{\mathfrak{s}}.
\]
In particular, $\deg(f_1)\leq \kappa_{\mathfrak{s}}/\log(2)$. Moreover, $d_3$ is also bounded by $\kappa_{\mathfrak{s}}$ so that 
\begin{equation}\label{case2:bd2} 
\deg(f)=\deg(f_1)\cdot d_3\leq \kappa_{\mathfrak{s}}^2/\log(2).
\end{equation} 
Likewise, $P=t^{d_1}$ implies that
\[d_1\, h_{K}^{\min}\leq d_1h(t)=h(t^{d_1})=h(P);\]
here we use that $t\in K$ is not a root of unity. In particular, $d_1\leq h(P)/h_{K}^{\min}$. Combining this fact with the bounds on \eqref{case2:bd} and \eqref{case2:bd2}, we see that 
\[\max\big\{\deg(\theta\circ f),\,\deg(\tau\circ g)\big\}=\max\big\{d_1\cdot\deg(f),d_2\big\}\leq \max\bigg\{\frac{\kappa_{\mathfrak{s}}^2 h(P)}{\log(2)\,h_{K}^{\min}},\;\kappa_{\mathfrak{s}}\bigg\}\leq d_{K,P}.\]
Here we use that $h(P)\geq h_{K}^{\min}$ and $\kappa_{\mathfrak{s}}>1$. This completes the proof of the claim in Case (2).               
\\[7pt] 
\textbf{Case (3):} Suppose that $\ell(f),\ell(g)\geq1$, so that we may write $f=\omega_1\circ f_1$ and $g=\omega_2\circ g_1$, where $\omega_1(z)=z^{e_1}+1$, $\omega_2(z)=z^{e_2}+1$, and $f_1,g_1\in M_{S_K}$. If $\tau=\theta$ then $d_1=d_2$ and $f(P)/g(P)$ is a $d_1$'st root of unity. But $d_1$ coprime to $w_K$ and $f(P)/g(P)\in K$, so that $f(P)=g(P)$. This is the first possible conclusion of Lemma \ref{lem:catalanbd}. 

Therefore, me may assume that $\tau\neq\theta$ and $d_1\neq d_2$ are distinct odd primes. In particular, there are integers $a_1$ and $a_2$ such that $a_1d_1+a_2d_2=1$. On the other hand \eqref{eq:relation} implies that $f(P)^{d_1}=g(P)^{d_2}$ from which it follows that
\begin{equation}\label{case3-catalan}
f_1(P)^{e_1}+1=f(P)=t^{d_2}\qquad\text{and}\qquad g_1(P)^{e_2}+1=g(P)=t^{d_1},  
\end{equation} 
where $t=f(P)^{a_2}\cdot g(P)^{a_1}$. Note that $f_1(P)$, $g_1(P)$ and $t$ are all $\mathfrak{s}$-integral since $P$ is $\mathfrak{s}$-integral, the coefficients of all relevant polynomials are integral, and the ring of $\mathfrak{s}$-integers is integrally closed (needed for $t$ only). Moreover $f_1(P)$, $g_1(P)$, $f(P)$, $g(P)$ cannot be roots of unity since, as argued in Case 2, Tate's telescoping Lemma and the fact that $h(P)>\log4$ together imply that each of these points has height at least $\log2$. There is nothing special about these functions; no points in $\Orb_{S_K}(P)$ can be roots of unity. Likewise $t$ cannot be a root of unity since otherwise $f(P)$, which equals $t^{d_2}$, must be. Therefore, \eqref{case3-catalan} implies that $(x,y,n,m)=(t,f_1(P),d_2,e_1)$ and $(x,y,n,m)=(t,g_1(P),d_1,e_2)$ are $\mathfrak{s}$-integral solutions to the Catalan equation. In particular, the bound in Theorem \ref{thm:catalanSintegers} implies that  
\begin{equation}\label{case3-catalan2}
\max\big\{h(t),h(f_1(P)),h(g_1(P)),d_1,d_2,e_1,e_2\big\}\leq\kappa_{\mathfrak{s}}.
\end{equation}
On the other hand, again applying Tate's telescoping Lemma to $f_1$, we deduce from \eqref{case3-catalan2} that 
\[\deg(f_1)\log2\leq\deg(f_1)(h(P)-\log2)=h(f_1(P))\leq\kappa_{\mathfrak{s}}.\]
In particular, $\deg(f_1)\leq \kappa_{\mathfrak{s}}/\log(2)$. Likewise, we may deduce that $\deg(g_1)\leq \kappa_{\mathfrak{s}}/\log(2)$. These facts coupled with \eqref{case3-catalan2} together imply that 
\begin{equation*} 
\begin{split} 
\deg(\theta\circ f)&=\deg(\theta\circ\omega_1\circ f_1)=d_1\cdot e_1\cdot \deg(f_1)\leq \frac{\kappa_{\mathfrak{s}}^3}{\log2} \\[5pt]
\deg(\tau\circ g)&=\deg(\tau\circ\omega_2\circ g_1)=d_2\cdot e_2\cdot \deg(g_1)\leq \frac{\kappa_{\mathfrak{s}}^3}{\log2}.
\end{split}  
\end{equation*}
In particular, $\max\{\deg(\theta\circ f),\deg(\tau\circ g)\}\leq \kappa_{\mathfrak{s}}^3/\log2\leq d_{K,P}$ by definition of $d_{K,P}$ as claimed.                
\end{proof} 
With Lemma \ref{lem:catalanbd} in place, we may define a few more constants depending only on the point $P\in K$ and and the number field $K$:   
\begin{equation}\label{catalnconstants}
t_{K,P}:=\#\{f\in M_{S_K}\,:\, \deg(f)\leq d_{K,P}\}\qquad \text{and} \qquad r_{K,P}:=\lceil\log_2\,d_{K,P}\rceil; 
\end{equation}   
note that there are only finitely functions of bounded degree in $M_{S_K}$. Hence, $t_{K,P}$ is well defined. From here, we are ready to prove the key lemma, which allows us to pass from counting functions to counting points in orbits; compare to \cite[Lemma 4.8]{Me:monoid}.  
\begin{lemma}\label{lem:pointstofunctionscatalan} Let $S_K$ be as in \eqref{SK} and suppose that $P\in\mathbb{P}^1(K)$ satisfies $h(P)>\log4$. Then there exists a constant $w_{K,P}$ (depending only on $K$ and $P$) such that
\[\#\{f\in M_{S_K}: f(P)=Q\}\leq w_{K,P}\]
holds for all $Q\in\Orb_{S_K}(P)$. Specifically, one can take $w_{K,P}=r_{K,P}\cdot t_{K,P}+1$. 
\end{lemma} 
\begin{proof}
First some notation. Given a point $Q\in\Orb_{S_K}(P)$ in the orbit of $P$, we define 
\[\ell_{P,Q}:=\min\{\ell(f): f(P)=Q\;\text{for}\;f\in M_{S_K}\}\] 
to be the minimum length of the functions evaluating $P$ to $Q$. Then the first step in the proof of Lemma \ref{lem:pointstofunctionscatalan} is to show that \vspace{.1cm} 
\begin{equation}\label{keybd1}
\qquad\#\{f\in M_{S_K}: f(P)=Q\}\leq \ell_{P,Q}\cdot t_{K,P}+1\qquad \text{for all $Q\in\Orb_{S_K}(P)$}. \vspace{.1cm} 
\end{equation}  
To do this, we proceed by induction on the minimum length $\ell_{P,Q}$. Note that if $\ell_{P,Q}=0$, then $P=Q$ and therefore any other function $f\in M_{S_K}$ evaluating $P$ to $Q$ must satisfy $f(P)=P$. But this is impossible for any non-identity $f\in M_{S_K}$ since $h(P)>\log4$ and
\begin{equation}\label{noperiodic}
\log2<3(h(P)-\log(2))-h(P)\leq \deg(f)(h(P)-\log(2))-h(P)\leq h(f(P))-h(P)
\end{equation} 
by Tate's Telescoping Lemma. Hence, the only function $f\in M_{S_K}$ with $f(P)=Q$ is the identity function when $\ell_{P,Q}=0$. In particular, \eqref{keybd1} holds as claimed. We now  proceed with the remaining cases and pro forma make $\ell_{P,Q}=1$ the base case of our induction argument.  \\[4pt]
\underline{Base case}: Suppose that $\ell_{P,Q}=1$ and choose some $\tau\in S_K$ with $\tau(P)=Q$. Now let $F\in M_{S_K}$ be any other function such that $F(P)=Q$. Then we may write $F=\theta\circ f$ for some $f\in M_{S_K}$ (possibly the identity function). Hence, letting $g$ be the identity function, we have that 
\[\theta\circ f(P)=\tau\circ g(P).\]
Therefore, Lemma \ref{lem:catalanbd} implies that either      
\[\text{$\theta=\tau$ and $f(P)=g(P)$\qquad or \qquad $\max\big\{\deg(\theta\circ f),\,\deg(\tau\circ g)\big\}\leq\,d_{K,P}$}.\]
However in the first case, $f(P)=g(P)=P$ so that $f$ must be the identity function also; see \eqref{noperiodic} in Case 1 above. Hence, $F=\tau$ in this case. In the other case, $\deg(F)=\deg(\theta\circ f)\leq d_{K,P}$. In particular, for all $Q\in\Orb_{S_K}(P)$ with $\ell_{P,Q}=1$, we have proven that there exists some $\tau=\tau_Q\in S_K$ satisfying \vspace{.1cm}  
\[\{F\in M_{S_K}\,: F(P)=Q\}\subseteq\{\tau_Q\}\cup\{F\in M_{S_K}\,:\, \deg(F)\leq d_{K,P}\}.\vspace{.1cm}\] 
In particular, \eqref{keybd1} holds as claimed in the $\ell_{P,Q}=1$ case. \\[7pt] 
\underline{Induction step}: Suppose that $\ell_{P,Q}=n+1$ and that \eqref{keybd1} holds for all $Q'\in\Orb_{S_K}(P)$ with $\ell_{P,Q'}\leq n$. Then we can choose (and fix) some $G=\tau\circ g$ satisfying $G(P)=Q$ with $\tau=\tau_Q\in S_K$ and $g\in M_{S_K}$ is of length $\ell(g)=n\geq1$. Now let $F\in M_{S_K}$ be any other function satisfying $F(P)=Q$. Then we may write $F=\theta\circ f$ for some $\theta\in S_K$ and some $f\in M_{S_K}$. In particular, Lemma \ref{lem:catalanbd} implies that either \vspace{.1cm}      
\begin{equation}\label{part1}
\text{$\theta=\tau$ and $f(P)=g(P)$\qquad or \qquad $\max\big\{\deg(\theta\circ f),\,\deg(\tau\circ g)\big\}\leq\,d_P$}. \vspace{.1cm}
\end{equation} 
Now let $Q'=g(P)$ so that we may rephrase \eqref{part1} as saying: \vspace{.1cm}   
\begin{equation}\label{part2}
\Scale[.92]{
\{F\in M_{S_K}\,: F(P)=Q\}\subseteq\{F\in M_{S_K}\,:\;F=\tau_Q\circ f \;\text{and}\; f(P)=Q'\}\cup\{F\in M_{S_K}\,:\, \deg(F)\leq d_{K,P}\}}. \vspace{.1cm} 
\end{equation} 
The point here is that $\tau_Q$ and $Q'$ are both independent of $F$ (they depend on $G$, which is fixed). On the other hand, $Q'\in\Orb_{S_K}(P)$ and $\ell_{P,Q'}\leq n$, so by induction \vspace{.1cm}
\begin{equation}\label{part3}  
\#\{f\in M_{S_K}\,:\,f(P)=Q'\}\leq\ell_{P,Q'}\cdot t_{K,P}+1\leq n\, t_{K,P}+1.
\end{equation} 
Hence, combining \eqref{part2} and \eqref{part3} we deduce that 
\[\#\{F\in M_{S_K}\,:\,F(P)=Q\}\leq (n\, t_{K,P}+1)+  t_{K,P}=(n+1)t_{K,P}+1=\ell_{P,Q}\cdot t_{K,P}+1\]
as desired. In particular, we have established \eqref{keybd1}, our first step in proving Lemma \ref{lem:pointstofunctionscatalan}.  

To complete the proof, we will show that if $\ell_{P,Q}$ is sufficiently large, then to count functions $F$ such that $F(P)=Q$, one can instead count functions $f$ with $f(P)=Q'$ and $\ell_{P,Q'}$ bounded. In particular combined with the first step \eqref{keybd1}, we obtain the desired general bound (independent of length). To do this, recall first that $r_{K,P}:=\lceil\log_2\,d_{K,P}\rceil$ and note that 
\begin{equation}\label{eq:catalanlength}
\deg(G)>2^{\ell(G)}\geq d_{K,P}\;\;\;\text{for all $G\in M_{S_K}$ with $\ell(G)\geq r_P$,}
\end{equation} 
since each non-identity map in $M_{S_K}$ has degree greater than $2$. Now suppose that $Q\in\Orb_{S_K}(P)$ is such that $\ell_{P,Q}>r_{K,P}$, and choose (and fix) a minimal length function $G$ such that $G(P)=Q$. Then, by the length assumption on $Q$, we may write $G=\tau_m\circ\dots\circ\tau_1\circ g$ for some $\tau_i\in S_K$ and $m\geq1$ and some $g\in M_{S_K}$ of length $r_{K,P}$. Similarly, due to the minimality of the length of $G$, we may write any other function $F$ with $F(P)=Q$ as $F=\theta_m\circ\dots\circ\theta_1\circ f$ for some $\tau_i\in S_K$ and $m\geq1$ and some $f\in M_{S_K}$ of length at least $r_{K,P}$. We'll show by induction on $m$ that 
\begin{equation}\label{claim}
\boxed{\text{\textbf{Claim:} $\tau_i=\theta_i$ for all $i$ and $f(P)=g(P)$}}\vspace{.1cm}
\end{equation}     
\underline{Base case:} Suppose that $m=1$. Then $F(P)=\theta_1\circ f(P)=\tau_1\circ g(P)=G(P)$. Hence, Lemma \ref{lem:catalanbd} implies that either 
\[\text{$\theta_1=\tau_1$ and $f(P)=g(P)$\qquad or \qquad $\max\big\{\deg(\theta_1\circ f),\,\deg(\tau_1\circ g)\big\}\leq\,d_{K,P}$}.\]
However, $\deg(\tau_1\circ g)=\deg(G)>d_{K,P}$ since $\ell(G)=1+r_{K,P}>r_{K,P}$; see \eqref{eq:catalanlength} above. In particular, $\theta_1=\tau_1$ and $f(P)=g(P)$ as claimed. \\[7pt] 
\underline{Induction step:} Suppose that $m\geq2$, that $\tau_{m-1}\circ\dots\circ\tau_1\circ g(P)=\theta_{m-1}\circ\dots\circ\theta_1\circ f(P)$ implies that $\tau_i=\theta_i$ for all $1\leq i\leq m-1$ and that $f(P)=g(P)$, and assume that  
\[\theta_{m}\circ\dots\circ\theta_1\circ f(P)=\tau_{m}\circ\dots\circ\tau_1\circ g(P).\] 
Now let $g'=\tau_{m-1}\circ\dots\circ\tau_1\circ g$ and $f'=\theta_{m-1}\circ\dots\circ\theta_1\circ f$. In particular, $\theta_m\circ f'(P)=\tau_m\circ g'(P)$. Therefore, Lemma \ref{lem:catalanbd} implies that either 
\[\text{$\theta_m=\tau_m$ and $f'(P)=g'(P)$\qquad or \qquad $\max\big\{\deg(\theta_m\circ f'),\,\deg(\tau_m\circ g')\big\}\leq\,d_{K,P}$}.\]  
But $\deg(\tau_m\circ g')>d_{K,P}$ since \[\ell(\tau_m\circ g')=1+\ell(g')\geq1+\ell(g)=1+r_{K,P}>r_{K,P};\] 
see \eqref{eq:catalanlength} above. In particular, $\theta_m=\tau_m$ and $f'(P)=g'(P)$. However, this means 
\[\tau_{m-1}\circ\dots\circ\tau_1\circ g(P)=\theta_{m-1}\circ\dots\circ\theta_1\circ f(P),\] 
which implies that $\tau_i=\theta_i$ for all $1\leq i\leq m-1$ and that $f(P)=g(P)$ by the induction hypothesis. Together with the fact that $\theta_m=\tau_m$, we have proven the claim \eqref{claim}.                
 
To summarize: suppose that $Q\in\Orb_{S_K}(P)$ is such that $\ell_{P,Q}>r_{K,P}$ and choose some minimal length $G$ with $G(P)=Q$. Now write $G=\tau_m\circ\dots\circ\tau_1\circ g$ for some $\tau_i\in S_K$ and $m\geq1$ and some $g\in M_{S_K}$ of length $r_{K,P}$. Then we have shown that if $F$ is \emph{any other function} satisfying $F(P)=Q$, then we may write 
\begin{equation}\label{reduction}
F=\tau_m\circ\dots\circ\tau_1\circ f\;\;\;\text{for some $f\in M_{S_K}$ satisfying $f(P)=g(P)$}. 
\end{equation} 
But $Q'=g(P)$ is in the orbit of $P$ and $\ell_{P,Q'}\leq\ell(g)=r_{K,P}$. Hence, \eqref{keybd1} implies that  
\[\#\{f\in M_{S_K}: f(P)=Q'\}\leq \ell_{P,Q'}\cdot t_{K,P}+1\leq r_{K,P}\cdot t_{K,P}+1.\]
Therefore, the number of such $F$'s in \eqref{reduction} is also bounded in this way: 
\[\#\{F\in M_{S_K}: F(P)=Q\}\leq r_{K,P}\cdot t_{K,P}+1;\] 
the key point is that both $Q'$ and the functions $\tau_i$ are independent of $F$. In particular, the bound $w_{K,P}:=r_{K,P}\cdot t_{K,P}+1$ satisfies 
\[\#\{F\in M_{S_K}: F(P)=Q\}\leq w_{K,P}\]
whenever $\ell_{P,Q}>r_{K,P}$. However, it also provides an upper bound when $\ell_{P,Q}\leq r_{K,P}$ by \eqref{keybd1}. This completes the proof of Lemma \ref{lem:pointstofunctionscatalan}. 
\end{proof} 
Finally, we are ready to prove Theorem \ref{thm:catalan} from the introduction: if $\mathfrak{q}$ is a uniformly-log discrete set of primes, then for all $\epsilon>0$ there is a $b=b(K,\mathfrak{q},\epsilon)$ such that 
\begin{equation}\label{eq:recap1}
(\log B)^{b}\ll\#\big\{Q\in \Orb_{S_{K,\mathfrak{q}}}(P)\,:\, H(Q)\leq B\big\}\ll (\log B)^{b+\epsilon}
\end{equation}
for all $P\in\mathbb{P}^1(K)$ with $H(P)>4$. 
\begin{proof}[(Proof of Theorem \ref{thm:catalan})] Let $\mathfrak{q}$ be a uniformly-log discrete set of primes. Then Corollary \ref{cor:examples} applied to $\mathfrak{d}=\mathfrak{q}$ and $c=1$ implies that for all $\epsilon>0$ there is a $b=b(K,\mathfrak{q},\epsilon)$ such that 
\begin{equation}\label{eq:recap2}
(\log B)^{b}\ll\#\big\{f\in M_{S_{K,\mathfrak{q}}}\,:\, H(f(Q))\leq B\big\}\ll (\log B)^{b+\epsilon}
\end{equation}
for all $P\in\mathbb{P}^1(K)$ with $H(P)>4$. On the other hand, it is certainly always true that 
\begin{equation}\label{eq:recap3}
\#\big\{Q\in \Orb_{S_{K,\mathfrak{q}}}(P)\,:\, H(Q)\leq B\big\}\leq\#\big\{f\in M_{S_{K,\mathfrak{q}}}\,:\, H(f(P))\leq B\big\}.
\end{equation} 
Moreover since $S_{K,\mathfrak{q}}\subseteq S_K$ (so that $M_{S_{K,\mathfrak{q}}}\subseteq M_{S_K}$), it follows from Lemma \ref{lem:pointstofunctionscatalan} that
\begin{equation}\label{eq:recap4}
w_{K,P}^{-1}\cdot\#\big\{f\in M_{S_{K,\mathfrak{q}}}\,:\, H(f(Q))\leq B\big\}\leq\#\big\{Q\in \Orb_{S_{K,\mathfrak{q}}}(P)\,:\, H(Q)\leq B\big\}. \vspace{.15cm}
\end{equation} 
In particular, \eqref{eq:recap2}, \eqref{eq:recap3}, \eqref{eq:recap4} together imply the desired bounds in \eqref{eq:recap1}.   
\end{proof}   
\begin{remark} In light of Lemmas \ref{lem:catalanbd} and \ref{lem:pointstofunctionscatalan}, it is tempting to think that Theorem \ref{thm:catalan} holds for the full semigroup generated by $z^q+1$ and $q$ is any prime (with or without the gcd condition). This may be so, however our bounds from section \ref{sec:countingfunctions} do not apply in this case, since the full set of rational primes do not form a uniformly log-discrete set. This leaves open the possibility that the poles of the generating functions we use accumulate, breaking down our estimates. That the primes are not uniformly log-discrete is perhaps deducible by elementary means. However, it certainly follows from Zhang's bounded gap theorem \cite{Zhang}. 
\end{remark} 
We can apply Theorem \ref{thm:catalan} to obtain an amusing arithmetic statement about the orbits generated by $z^q+1$ where $q$ is a Mersenne prime; see Corollary \ref{cor:amusing} from the introduction for an equivalent, but less dynamical looking, statement. 
\begin{corollary}\label{cor:catalan} Let $S=\big\{z^q+1:\,q\;\text{is a Mersenne prime}\big\}$. Then \vspace{.2cm}
\begin{equation*}
(\log B)^{0.60839}\ll\#\big\{Q\in \Orb_{S}(P)\,:\, H(Q)\leq B\big\}\ll (\log B)^{0.60872}\vspace{.2cm}
\end{equation*} 
for all $P\in\mathbb{P}^1(\mathbb{Q})\mysetminus\Big\{\infty,0,\pm{1},\pm{2},\pm{\frac{1}{2}},\pm{3},\pm{\frac{1}{3}},\pm{\frac{2}{3}},\pm{\frac{3}{2}},\pm{4},\pm{\frac{1}{4}},\pm{\frac{3}{4}}\Big\}$.  
\end{corollary} 
\begin{proof} It is straightforward to see that the Mersenne primes are uniformly log-discrete. For instance, $(2^{x+1}-1)/(2^x-1)>1.5$ for all $x>1$, so that
\[\big|\log(q)-\log(q')\big|>\log(1.5)\]
for all distinct Mersenne primes $q$ and $q'$. In particular, $\#\big\{Q\in \Orb_{S}(P)\,:\, H(Q)\leq B\big\}$ is bounded by arbitrarily close powers of $\log(B)$ for all $P$ with $H(P)>4$ by Theorem \ref{thm:catalan}. To obtain the explicit bounds in Corollary \ref{cor:catalan} (whose powers of $\log(B)$ agree up to three decimal places), we use the approximation techniques discussed in Remark \ref{rem:explicit}, the first nine Mersenne primes:  
\[q=3,\,7,\,31,\,127,\, 8191,\,131071,\,524287,\, 2147483647,\, 2305843009213693951,\]  
and the fact that the tenth has $27$ digits. See the file called Mersenne at  
\[{\tt{https://sites.google.com/a/alumni.brown.edu/whindes/research}}\]
for the Magma code justifying these bounds.   
\end{proof} 
\begin{example} To help illustrate Corollary \ref{cor:catalan} in a more classically arithmetic way, consider the statement for $P=5$. Then \vspace{.1cm}
\[(\log B)^{0.60839}\ll\#\bigg\{t=((5^{q_1}+1)^{q_2}\dots +1)^{q_s}+1\;\Big\vert\; \text{$q_1,\dots, q_s\in\mathfrak{m}$, $t\leq B$}\bigg\}\ll (\log B)^{0.60872}\;, \vspace{.1cm}\]
where $\textbf{m}$ is the set of Mersenne primes.   
\end{example}

 \vspace{.2cm} 
\end{document}